\documentclass[12pt]{article}

\usepackage{amsmath}
\usepackage{mathtools}
\usepackage{amssymb}
\usepackage{enumerate}
\usepackage{here}
\usepackage{fullpage}
\usepackage[usenames,dvipsnames,svgnames,table]{xcolor}
\usepackage{mathrsfs}
\usepackage{multirow}

\usepackage{cite}
\def\citep{\cite}

\usepackage{dsfont}

\usepackage{newfloat}
\DeclareFloatingEnvironment[
    fileext=lob,
    listname=List of Boxes,
    name=Box,
    placement=tbhp,
]{mybox}

\def\Tmin{T_{\textnormal{min}}}
\def\Tmax{T_{\textnormal{max}}}
\def\rmax{\rho_{\textnormal{max}}}

\def\T{\textnormal{T}}

\definecolor{lblue}{RGB}{0,110,152}

\definecolor{dred}{RGB}{171,67,53}

\providecommand{\blue}[1]{\color{black}{#1}\color{black}\hspace{0pt}}
\providecommand{\bblue}[1]{\color{black}{#1}\color{black}\hspace{0pt}}
\providecommand{\bl}[1]{\color{black}{#1}\color{black}\hspace{0pt}}
\providecommand{\bll}[1]{\color{black}{#1}\color{black}\hspace{0pt}}

\providecommand{\red}[1]{\color{red}{#1}\color{black}\hspace{0pt}}

\newtheorem{theorem}{Theorem}
\newtheorem{corollary}[theorem]{Corollary}
\newtheorem{proposition}[theorem]{Proposition}

\newtheorem{define}[theorem]{Definition}
\newtheorem{example}[theorem]{Example}
\newtheorem{assumption}[theorem]{Assumption}
\newtheorem{remark}[theorem]{Remark}

\newcommand{\mendth}{\hfill \ensuremath{\vartriangle}}

\DeclareMathOperator{\He}{Sym}
\DeclareMathOperator*{\diag}{diag}
\DeclareMathOperator{\eps}{\varepsilon}

\newenvironment{proof}{{\it Proof :~}}{\hfill$\diamondsuit$\\}

\def\T{\intercal}

\begin{document}
\title{Stability analysis and stabilization of LPV systems with jumps and (piecewise) differentiable parameters using continuous and sampled-data controllers}

\author{Corentin Briat\thanks{corentin@briat.info; http://www.briat.info}}
\date{}

%

\maketitle

\begin{abstract}
Linear Parameter-Varying (LPV) systems with jumps and piecewise differentiable parameters is a class of hybrid LPV systems for which no tailored stability analysis and stabilization conditions have been obtained so far\footnote{Except of course in the conference version of this work.}. We fill this gap here by proposing an approach based on a clock- and parameter-dependent Lyapunov function yielding stability conditions under both constant and minimum dwell-times. Interesting adaptations of the latter result consist of a minimum dwell-time stability condition for uncertain LPV systems and LPV switched impulsive systems. The minimum dwell-time stability condition is notably shown to naturally generalize and unify the well-known quadratic and robust stability criteria all together. Those conditions are then adapted to address the stabilization problem via timer-dependent and a timer- and/or parameter-independent (i.e. robust) state-feedback controllers, the latter being obtained from a relaxed minimum dwell-time stability condition involving slack-variables. Finally, the last part addresses the stability of LPV systems with jumps under a range dwell-time condition which is then used to provide stabilization conditions for LPV systems using a sampled-data state-feedback gain-scheduled controller. The obtained stability and stabilization conditions are all formulated as infinite-dimensional semidefinite programming problems which are then solved using sum of squares programming. Examples are given for illustration.

\end{abstract}

\section{Introduction}

 \bll{Impulsive systems are an important class of hybrid systems that can be used to represent a wide variety of real-world processes such as systems with impacts and robotic systems \cite{Bainov:89,Brogliato:00}, sampled-data systems \cite{Sivashankar:94,Goebel:12}, networked control systems \cite{Antunes:13}, vaccination in epidemiological networks \cite{Briat:09h}, biological reaction networks \cite{Anderson:15}, etc. Theoretical results pertaining to this class of systems have already been obtained such as the (finite-time) stability analysis and control of impulsive LTI and timer-dependent systems \cite{Briat:13d,Lawrence:20}, impulsive LTV systems \cite{Medina:09,Amato:14,Amato:16,Briat:19:Linf}, stochastic LTI impulsive systems \cite{Antunes:09,Antunes:09b,Antunes:13}, and general nonlinear systems \cite{Hespanha:08,Michel:08,Goebel:12}. Most approaches are based on the use of Lyapunov functions \cite{Hespanha:08}, Lyapunov functionals \cite{Naghshtabrizi:08}, looped-functionals \cite{Briat:11l} and time- or timer-dependent Lyapunov functions \cite{Medina:09,Goebel:12,Briat:13d} but other methods also exist such as vector Lyapunov functions \cite{Lakshmikantham:91}. In particular, the concept of \emph{dwell-time}, initially introduced for the analysis of switched systems \cite{Morse:96,Hespanha:99,Geromel:06b}, plays an essential role for establishing interesting stability results for impulsive systems subject to families of sequences of impulse instants instead of a single one. Dwell-times are simply the times between two consecutive impulse instants which may be constrained in order to define families of impulse instants. For instance, the minimum dwell-time constraint imposes a lower bound on the dwell-times whereas the range dwell-time constraint specifies that the dwell-time values belong to some known interval.}

 \bll{On the other hand, Linear Parameter-Varying (LPV) systems \citep{Shamma:88phd,Mohammadpour:12,Briat:book1} have been extensively studied over the last 30 years as they can model linear systems that intrinsically depend on time-varying parameters \citep{Wu:95} and approximate nonlinear systems \citep{Shamma:88phd,Shamma:92}. More importantly, they provide a natural framework for the systematic design of gain-scheduled controllers using, for instance, an adaptation of robust control methods \cite{Packard:94a,Apkarian:95a}. For those reasons, they have found applications in various fields such as automotive suspensions systems \citep{Poussot:10}, aeroelastic control  \citep{Seiler:12,Pfifer:15}, aperiodic sampled-data systems \citep{Robert:10}; see also \cite{Mohammadpour:12,Briat:book1,Hoffmann:15} for more examples. Over the past recent years, this field has been considerably consolidated by numerous theoretical approaches and numerical tools \citep{Packard:94a,Apkarian:95a,Apkarian:98a,Wu:01,Wu:06b,Scherer:01,Scherer:06,Scherer:12,Scherer:15,Briat:book1,Briat:15d,Balas:15}. On a more conceptual perspective, and as discussed in \citep{Briat:book1,Briat:15d}, LPV systems are often separated into two main and opposed categories depending on whether the parameter trajectories are continuously differentiable or can vary arbitrarily fast, possibly including discontinuities. This strict separation gave rise to the concepts of robust and quadratic stability, where the former considers a parameter-dependent Lyapunov function \cite{Wu:95,Apkarian:98a} and the latter a parameter-independent one \cite{Becker:94,Apkarian:95a}. But this classification is very crude since these two families are quite extreme and are clearly unadapted to deal with piecewise differentiable parameter trajectories. Indeed, such trajectories do not belong to the first category because of the presence of discontinuities. Moreover, while those trajectories technically belong to the second category, considering a quadratic stability criterion would result in too conservative results because such an approach would miss the fact that the parameter trajectories actually have bounded derivatives between discontinuities. This remark motivated the consideration of LPV systems with piecewise constant parameters in \citep{Briat:15d} in order to demonstrate the benefits of using a more accurate description of the parameter trajectories for the establishing the stability of LPV systems for which a robust stability analysis would have been inapplicable and a quadratic stability analysis proven inconclusive. Those results were obtained using an adaptation of the approaches in \citep{Briat:13d,Briat:14f,Briat:15i}, originally developed for switched and impulsive systems, and the use of the concept of dwell-time, measuring in this case the time between two consecutive discontinuities in the parameter trajectories. A particularly striking example was a system which was not quadratically stabilizable but which became stabilizable as long as a lower bound was imposed on the dwell-times. This is an important remark as this shows that neglecting even a small gap between consecutive discontinuities, a very realistic and practical assumption, may lead to an incorrect assessment of the properties of a system, here its stabilizability.}

\bll{One of the objectives of this paper (and part of the main objective of its conference version \cite{Briat:17ifacLPV}) is to extend those results to the case of LPV systems with piecewise differentiable parameters. Here, we go beyond the conference version of the paper by considering LPV impulsive systems with piecewise differentiable parameters. Such parameter trajectories may arise when an impulsive LPV system is obtained as an approximation of a nonlinear impulsive system, or simply when the parameters naturally have such a behavior \cite{Joo:14}. They can also be used to approximate parameter trajectories that exhibit intermittent very fast, yet smooth, variations. It is worth mentioning here that impulsive LPV systems have not yet been fully addressed in the literature until now. Using a clock- and parameter-dependent Lyapunov function, both the cases of periodic and aperiodic jumps in the state and the parameter trajectories are considered by relying on the use of the concepts of constant and minimum dwell-times. A relaxed result is also provided in order to derive stability conditions for uncertain impulsive LPV systems subject to polytopic uncertainties and for switched impulsive LPV systems. We then prove that the well-known quadratic stability and robust stability conditions can be recovered from the minimum dwell-time stability conditions: quadratic stability is recovered when the lower bound on the dwell-time tends to zero while robust stability is retrieved when this bound goes to infinity. A similar result was obtained in \cite{Briat:15d} in the context of piecewise constant parameters.}

\bll{Those results are then exploited to derive convex conditions for the design of gain-scheduled controllers for LPV systems with piecewise differentiable parameters, a problem that has never been addressed so far. A possible limitation of the approach is that it leads to timer-dependent controllers, which may be difficult to implement in practice, in particular if the gain of the controller varies a lot over short time-scales. To remedy this issue, the relaxed stability result previously derived for the analysis of uncertain systems is used to provide stabilization conditions using a timer-independent controller, a result which is also novel. A minor modification of this result can also be considered for the design of parameter-independent (i.e. robust) controllers.}

\bll{Differently from the conference version of this paper, the consideration impulsive LPV systems allows us to address the problem of designing sampled-data gain-scheduled state-feedback controllers for LPV systems due to the fact that a sampled-data system can be exactly represented as an impulsive system \cite{Sivashankar:94}. The sampled-data control problem of LPV systems has been considered in the past in \citep{Tan:02} using a discretization approach (assuming the parameters are piecewise constant), in \citep{Ramezanifar:12} using the so-called input-delay approach \cite{Fridman:04}, and in \citep{GomesdaSilva:15,Gomes:18} using looped-functionals. In the present paper, the approach is based on the use of a parameter- and clock-dependent Lyapunov function and yields convex stability conditions under a range dwell-time constraint. This result is then exploited to provide convex design conditions for the sampled-data control of LPV systems, a result that is more difficult to obtain in the context of looped- or Lyapunov-functionals, due to the presence of a larger number of decision variables leading to complex non-convex terms that are difficult to deal with efficiently. The price to pay, however, is that the stability and design conditions are stated in terms of infinite-dimensional semidefinite programs \citep{Briat:13d,Briat:14f,Briat:15i} which cannot be directly solved and need to be converted into finite-dimensional problems first. In this paper, we rely on Sum-of-Squares (SOS) programming \citep{Parrilo:00,sostools3} that transforms the original infinite-dimensional problem into a (possibly very large) finite-dimensional semidefinite program \citep{Boyd:04} under the assumption that the conditions are polynomial in the timer and the parameter variables.}\\


\noindent \textbf{Outline.} \bll{The structure of the paper is as follows: in Section \ref{sec:prel} preliminary definitions and results are given. Section \ref{sec:stab} develops the stability analysis results under constant and minimum dwell-times. The latter result is then extended to address the case of uncertain systems and impulsive switched LPV systems. Some correspondence with existing results in the literature are also provided together with a procedure detailing on how to solve the infinite-dimensional feasibility problems using SOS programming. The stability results of Section \ref{sec:stab} are then extended to control design in Section \ref{sec:stabz1}. Finally, Section \ref{sec:stabz2} is devoted to the analysis of impulsive LPV systems under range dwell-time and to the application of this result to the sampled-data control of LPV systems. The examples are treated in the related sections.}

\noindent \textbf{Notations.} The set of nonnegative and positive integers are denoted by $\mathbb{Z}_{\ge0}$ and $\mathbb{Z}_{>0}$, respectively. The set of symmetric matrices of dimension $n$ is denoted by  $\mathbb{S}^n$ while the cone of positive (semi)definite matrices of dimension $n$ is denoted by ($\mathbb{S}^n_{\succeq0}$) $\mathbb{S}_{\succ0}^n$. For some $A,B\in\mathbb{S}^n$, the notation that $A\succ(\succeq)B$ means that $A-B$ is positive (semi)definite. 
For a square matrix $A$, we define the shorthand $\He[A]=A+A^{\T}$. For any differentiable function $f(x,y)$, the partial derivatives with respect to the first and second argument evaluated at $(x,y)=(x^*,y^*)$ are denoted by $\partial_x f(x^*,y^*)$ and $\partial_y f(x^*,y^*)$, respectively. For a function $f$, the right-handed limit at a point $t$ in its domain is defined as $\textstyle f(t^+)=\lim_{s\downarrow t}f(s)$.

\section{Preliminaries}\label{sec:prel}

LPV systems with \bblue{state-}jumps are dynamical systems which can be described as
\begin{equation}\label{eq:mainsyst}
\begin{array}{rcl}
    \dot{x}(t)&=&A(t-t_k,\rho(t))x(t),\ t\in(t_k,t_{k+1}],k\in\mathbb{Z}_{\ge0}\\
    x(t_k^+)&=&J(\rho(t_k))x(t_k),k\in\mathbb{Z}_{>0}\\
    x(0^+)&=&x(0)=x_0,\ t_0=0
\end{array}
\end{equation}
where $x,x_0\in\mathbb{R}^n$ are the state of the system and the initial condition, respectively. \bll{The trajectories of the above system are left-continuous; i.e. $\textstyle\lim_{s\uparrow t_k}x(s)=x(t_k)$ and $\textstyle\lim_{s\downarrow t_k}x(s)=x(t_k^+)$. The trajectories of the parameters $\rho(\cdot)$ are assumed to be piecewise differentiable and such that (i) $\rho:\mathbb{R}_{\ge0}\to\mathcal{P}\subset\mathbb{R}^N$, where $\mathcal{P}$ is compact and connected, and (ii) $\dot{\rho}\in\mathcal{Q}(\rho)\subset\mathcal{D}\subset\mathbb{R}^N$ everywhere $\dot{\rho}$ exists. The set $\mathcal{Q}(\rho)$ contains all the possible values for the parameter derivatives when the value of the parameters is $\rho$. This set is important as it allows to prevent the parameters from leaving the set $\mathcal{P}$. On the other hand, the set $\mathcal{D}$ is a set that contains all the possible values for the parameter derivatives, irrespectively of the current value of the parameters. Considering $\dot{\rho}\in\mathcal{D}$ would technically allow the parameters to leave $\mathcal{P}$ -- an example is given later. The matrix-valued functions $A:\mathbb{R}_{\ge0}\times\mathcal{P}\mapsto\mathbb{R}^{n\times n}$ and $J:\mathcal{P}\mapsto\mathbb{R}^{n\times n}$ are assumed to be continuous. Some additional assumptions will sometimes be considered and they will be mentioned when this is the case. The state discontinuities only occur at the times $t_k$, $k\in\mathbb{Z}_{>0}$  and no discontinuity occurs at $t_0=0$. We make the following assumptions on the sequence $\{t_k\}_{k\in\mathbb{Z}_{\ge0}}$ of impulse times: (i) it is increasing and (ii) it grows without bound; i.e. $t_{k+1}>t_k$ and $t_k\to\infty$ as $k\to\infty$. The \emph{dwell-times} are defined as $T_k:=t_{k+1}-t_k$ where $k\in\mathbb{Z}_{\ge0}$.} \bll{The following assumption will sometimes be used in order to simplify the exposition of the results:
\begin{assumption}\label{hyp:param}
When the parameters are independent of each other, then we can safely assume that $\mathcal{P}$ is a box and that the following decompositions hold $\mathcal{P}=\mathcal{P}_1\times\ldots\times\mathcal{P}_N$ where $\mathcal{P}_i:=[\underline{\rho}_i,\ \bar{\rho}_i]$, $\underline{\rho}_i\le\bar{\rho}_i$ and $\mathcal{D}=\mathcal{D}_1\times\ldots\times\mathcal{D}_N$ where $\mathcal{D}_i:=[\underline{\nu}_i,\ \bar{\nu}_i]$, $\underline{\nu}_i<0<\bar{\nu}_i$. We also define the set of vertices of $\mathcal{D}$ as $\mathcal{D}^v$; i.e. $\mathcal{D}^v:=\{\underline{\nu}_1,\ \bar{\nu}_1\}\times\ldots\{\underline{\nu}_N,\ \bar{\nu}_N\}$. Moreover, in such a case, we have that $\mathcal{Q}(\rho)=\mathcal{Q}_1(\rho)\times\ldots\times\mathcal{Q}_N(\rho)$ together with
\begin{equation}
  \mathcal{Q}_i(\rho):=\left\{\begin{array}{ccl}
    \mathcal{D}_i&&\textnormal{if }\rho_i\in(\underline{\rho}_i,\bar{\rho}_i),\\
    \mathcal{D}_i\cap\mathbb{R}_{\ge0}&&\textnormal{if }\rho_i=\underline{\rho}_i,\\
    \mathcal{D}_i\cap\mathbb{R}_{\le0}&&\textnormal{if }\rho_i=\bar{\rho}_i.
  \end{array}\right.
\end{equation}
\end{assumption}
It is important to stress that this assumption is not necessary and all the results can be vert generally stated; e.g. when $\mathcal{P}$ is defined by polynomial inequalities and equalities. This also illustrates the difference between $\dot{\rho}\in\mathcal{D}$ and $\dot{\rho}\in\mathcal{Q}(\rho)$ where we can see that the second inclusion prevents the parameters to leave $\mathcal{P}$.}

It is convenient to introduce here the following discrete-time system associated with the system \eqref{eq:mainsyst}:
\begin{define}
  The \emph{pre-jump embedded discrete-time system} associated with the system \eqref{eq:mainsyst} is given by
  \begin{equation}\label{eq:mainsystDT1}
\begin{array}{rcl}
    x(t_{k+1})&=&M_kx(t_k),k\ge0\\
    x(t_0)&=&x_0
\end{array}
\end{equation}
where $M_k\in\mathcal{M}(T_k,\rho(t_k))$ and
\begin{equation}
\mathcal{M}(T,\varpi):=\left\{\Psi(T)J_k(\varpi)\left|\begin{array}{l}
\dot{\Psi}(\tau)=A(\tau,\varrho(\tau))\Psi(\tau),\Psi(0)=I,\\
\varrho(\tau)\in\mathcal{P}, \dot{\varrho}(\tau)\in\mathcal{Q}(\varrho(\tau)),\tau\in[0,T]
\end{array}\right.\right\},\label{eq:Mkm}
\end{equation}
where $J_0(\cdot)=I$ and $J_k(\varpi)=J(\varpi)$, $k\ge1$.
\end{define}

\bll{Due to the fact that the parameters admit an infinite number of possible trajectories, the discrete-time system above can be understood as an uncertain system. Indeed, the state matrix $M_k$ of the system \eqref{eq:mainsystDT1} can be readily seen to belong to the set $\mathcal{M}_k(T_k,\rho(t_k))$ which contains all the possible maps $x(t_k)\mapsto x(t_{k+1})$ for every parameter trajectories obeying the boundedness and differentiability assumptions, and starting from the value $\rho(t_k)$. Needless to say that this set is very complex and cannot be explicitly computed except, perhaps, in some very special or irrelevant cases. In this regard, establishing the stability of the discrete-time system \eqref{eq:mainsystDT1} directly from its discrete-time formulation is a formidable task. Some methods have been developed to address or, in fact, circumvent this problem. In particular, the looped-functional framework initiated in \cite{Seuret:12} and adapted to impulsive systems in \cite{Briat:11l} has been shown to overcome this difficulty by providing implicit discrete-time stability conditions disguised as continuous-time ones. This approach had the benefits of circumventing the need for the direct computation of state-transition matrices, which made the overall framework readily  applicable to uncertain and nonlinear systems. Another approach, which will be considered in this paper, relies on clock-dependent (or timer-dependent) Lyapunov functions which achieve exactly the same objectives as looped-functionals but with the extra feature of allowing for an immediate derivation of convex design conditions; see e.g. \cite{Briat:15f} for more details.}

\section{Stability analysis of LPV systems with jumps and piecewise differentiable parameters}\label{sec:stab}

\bll{The objective of this section is to present the main stability analysis results for LPV impulsive systems with piecewise continuous parameters together with some related discussions. Section \ref{sec:cstDT} presents the result pertaining to the constant dwell-time case whereas Section \ref{sec:minDT} addresses the minimum dwell-time case. This latter result is then generalized to the polytopic uncertain system case in Section \ref{sec:minDT2} and to impulsive switched LPV systems in Section \ref{sec:switched}. The connection between the minimum dwell-time stability result and the concepts of quadratic and robust stability is clarified in Section \ref{sec:QR}. Computational considerations are discussed in Section \ref{sec:computational} while examples are treated in Section \ref{sec:examples}.}

\subsection{Stability under constant dwell-time}\label{sec:cstDT}

We consider in this section the following family of periodically changing piecewise differentiable parameter trajectories
\begin{equation}\label{eq:paramcDT}
  \mathscr{P}_{\hspace{-2pt}\scriptscriptstyle{\bar{T}}}:=\left\{\begin{array}{c}
    \rho:\mathbb{R}_{\ge0}\mapsto \mathcal{P}\left|\begin{array}{c}
      \dot{\rho}(t)\in\mathcal{Q}(\rho(t)),t\in(t_k,t_{k+1}],T_k=\bar T,\\
    t_0=0,\rho(t_0)=\rho(t_0^+),k\in\mathbb{Z}_{\ge0}
    \end{array}\right.\end{array}\right\},
\end{equation}
\bll{where $\bar T>0$, and $T_k:=t_{k+1}-t_k$ is the dwell-time defined in the previous section. In other words, the trajectories contained in this family can only exhibit jumps at the times $t_k=k\bar{T}$, $k\in\mathbb{Z}_{>0}$ and, hence, the distance between two potential successive state and parameter discontinuities is constant, whence the name \emph{constant dwell-time}. Although not the most interesting case, the constant dwell-time case is interesting to set-up the main ideas before addressing more complicated cases. It can also be useful to treat the case of periodic parameter systems with periodic discontinuities as also considered in \cite{Briat:15f}. Finally, it seems interesting to note that, in the present case, we have that $\mathcal{M}(T_k,\rho)=\mathcal{M}(T_0,\rho)$, $k\in\mathbb{Z}_{\ge0}$. This shows that the discrete-time system can be considered as a discrete-time LPV system scheduled with the parameter $\rho(t_k)$.} We have the following result:
\begin{theorem}[Constant dwell-time]\label{th:cstDT}
\bll{Assume that the parameter trajectories satisfy Assumption \ref{hyp:param} and let $\bar T\in\mathbb{R}_{>0}$ be given.} Assume further that there exist a bounded continuously differentiable matrix-valued function $P:[0, \bar{T}]\times\mathcal{P}\mapsto \mathbb{S}^n$, $P(\bar T,\rho)\in\mathbb{S}^n_{\succ0}$, $\rho\in\mathcal{P}$, and a scalar $\eps>0$ such that the conditions
     \begin{equation}\label{eq:cst:1}
      \partial_\tau P(\tau,\rho)+\sum_{i=1}^N\partial_{\rho_i} P(\tau,\rho)\mu_i+\He[P(\tau,\rho)A(\tau,\rho)]\preceq0
    \end{equation}
    and
     \begin{equation}\label{eq:cst:2}
      J(\rho)^{\T}P(0,\rho^+)J(\rho)-P(\bar{T},\rho)+\eps I_n\preceq0
    \end{equation}
    hold for all $(\tau,\rho,\rho^+,\mu)\in[0, \bar{T}]\times\mathcal{P}\times\mathcal{P}\times \mathcal{D}^v$. Then, the LPV system \eqref{eq:mainsyst} with parameter trajectories in $\mathscr{P}_{\hspace{-2pt}\scriptscriptstyle{\bar{T}}}$ is uniformly exponentially stable.

\bll{Moreover, when the above conditions hold, then there exists a matrix-valued function $Q:\mathcal{P}\mapsto\mathbb{S}^n_{\succ0}$ such that
\begin{equation}
  M^{\T}Q(\rho^+)M-Q(\rho)\prec0
\end{equation}
holds for all $M\in\mathcal{M}(\bar T,\rho)$ and all $(\rho,\rho^+)\in\mathcal{P}\times\mathcal{P}$ and the pre-jump embedded discrete-time system \eqref{eq:mainsystDT1} is uniformly exponentially stable.}\hfill\mendth
\end{theorem}
\bll{\begin{proof}
Let us consider the function $V(x,\tau,\rho)=x^{\T}P(\tau,\rho)x$ where $P$ is as defined in the result. From the definition of $P$, the compactness of $\mathcal{P}$, and the inequality \eqref{eq:cst:2}, there exist some scalars $\alpha_1,\alpha_2>0$ such that
\begin{equation}\label{eq:hdsjkhdkshdkshkjdhskdj}
  \alpha_1||x||_2^2\le V(x,\bar T,\rho)\le \alpha_2||x||_2^2
\end{equation}
 holds for all $\rho\in \mathcal{P}$. Note also that, by virtue of the convexity of the conditions, the condition \eqref{eq:cst:1} is feasible for all $(\tau,\rho,\rho^+,\mu)\in[0, \bar{T}]\times\mathcal{P}\times\mathcal{P}\times \mathcal{D}^v$ if and only if it is feasible for all $(\tau,\rho,\rho^+,\mu)\in[0, \bar{T}]\times\mathcal{P}\times\mathcal{P}\times \mathcal{D}$. Pre- and post-multiplying \eqref{eq:cst:1} by $x(t_k+\tau)^{\T}$ and $x(t_k+\tau)$, letting $(\rho,\mu)\leftarrow (\rho(t_k+\tau),\dot{\rho}(t_k+\tau))$, and integrating from 0 to $\bar T$ with respect to $\tau$ yields
\begin{equation}\label{eq:45646546546546dsdsd5}
  V(x(t_{k+1}),\bar T,\rho(t_{k+1}))-V(x(t_{k}^+,0,\rho(t_{k}^+)))\le0.
\end{equation}
Using now the fact that $x(t_{k+1})=\Phi_\rho(t_{k+1},t_k)x(t_k^+)$, we get
\begin{equation}
x(t_k^+)^{\T}\left(\Phi_\rho(t_{k+1},t_k)^{\T}P(\bar T,\rho(t_{k+1}))\Phi_\rho(t_{k+1},t_k)-P(0,\rho(t_{k}^+))\right)x(t_k^+)\le 0.
\end{equation}
As the above inequality holds for all $x(t_k^+)\in\mathbb{R}^n$, this is equivalent to saying that
\begin{equation}\label{eq:kjopDSkos;djsdopk;j}
\Phi_\rho(t_{k+1},t_k)^{\T}P(\bar T,\rho(t_{k+1}))\Phi_\rho(t_{k+1},t_k)-P(0,\rho(t_{k}^+))\preceq0.
\end{equation}
Considering now \eqref{eq:cst:2} with $\rho^+=\rho(t_{k}^+)$ and $\rho=\rho(t_{k})$ yields
  \begin{equation}
      J(\rho(t_{k}))^{\T}P(0,\rho(t_{k}^+))J(\rho(t_{k}))+\eps I_n\preceq P(\bar{T},\rho(t_{k}))
    \end{equation}
 which together with \eqref{eq:kjopDSkos;djsdopk;j} imply that
\begin{equation}\label{eq:hiddenDT:cstDT}
M_k^{\T}P(\bar T,\rho(t_{k+1}))M_k-P(\bar T,\rho(t_{k}))\preceq-\eps I
\end{equation}
holds for all $M_k\in\mathcal{M}(\bar T,\rho(t_k))$ and all $(\rho(t_k),\rho(t_{k+1}))\in\mathcal{P}\times\mathcal{P}$. Since $P(\bar T,\cdot)$ is positive definite, the above inequality implies that the pre-jump embedded discrete-time system \eqref{eq:mainsystDT1} is uniformly exponentially stable. Indeed, letting $V_{k}:=V(x(t_{k}),\bar T,\rho(t_{k}))$, then we have that $V_{k+1}-V_k\le-\eps||x(t_k)||_2^2$ which, together with \eqref{eq:hdsjkhdkshdkshkjdhskdj}, imply that
\begin{equation}
  V_{k+1}\le \psi V_k
\end{equation}
where $\psi:=1-\eps/\alpha_1\in(0,1)$. Therefore, we have that $V_k\le\psi^kV_0$ and that $||x(t_k)||_2^2\le \textstyle\frac{\alpha_2}{\alpha_1}\psi^k||x_0||_2^2$. This proves that $||x(t_k)||\to0$ as $k\to\infty$.

To show the uniform exponential stability of the system, it is enough to use the fact that the function $V$ is nonincreasing over each interval $(t_k,t_{k+1}]$. Hence, we have that
\begin{equation}\label{eq:jdksjldksjdksdldjsjdlk}
  V_{k+1}\le V(x(t_k+\tau),\tau,\rho(t_k+\tau))\le V(x(t_k^+),0,\rho(t_k^+))\le V_k, \tau\in(0,T_k]
\end{equation}
where the last inequality comes from \eqref{eq:cst:2}. This proves that $V(x(t_k+\tau),\tau,\rho(t_k+\tau))$ is positive for all $\tau\in(0,\bar T]$ and all $k\ge0$, and that since $V_k\to0$ as $k\to\infty$, then so is $\textstyle\sup_{\tau\in(0,T_k]} V(x(t_k+\tau),\tau,\rho(t_k+\tau))$. Therefore, the impulsive LPV system \eqref{eq:mainsyst} with parameter trajectories in $\mathscr{P}_{\hspace{-2pt}\scriptscriptstyle{\bar{T}}}$ is uniformly exponentially stable under constant dwell-time $\bar T$. The final statement of the result can be easily proven by considering \eqref{eq:hiddenDT:cstDT} with $Q(\cdot)=P(\bar T,\cdot)$.
\end{proof}}

\bll{This result deserves a few comments. First of all, the first condition implies that the function $V(x,\tau,\rho)$ is nonincreasing along the flow of the system whereas the second one implies that it is decreasing at jumps. Another point is that we only require the Lyapunov function to be positive definite at $\tau=\bar T$ because this condition implies the positivity of the function for all timer values, as shown in \eqref{eq:jdksjldksjdksdldjsjdlk}. This may seem anecdotic at first sight but requiring that $P(\bar T,\rho)$ to be positive definite for all $\rho\in\mathcal{P}$ is computationally much less expensive than demanding that $P(\tau,\rho)$ be positive definite for all $(\tau,\rho)\in[0,\bar T]\times\mathcal{P}$. The same remark was also raised in \cite{Briat:13d}. Another comment is that the conditions in the above result formulate, in fact, a robust discrete-time stability conditions for the pre-jump embedded discrete-time system \eqref{eq:mainsystDT1} as stated in the concluding statement of the result. Finally, it is interesting to point out that, as long as stability is concerned, the result remains valid for all sequences of dwell-times which are eventually constant; i.e. for which there exists a $\kappa\in\mathbb{Z}_{\ge0}$ such that $T_k=T_\kappa$ for all $k\ge\kappa$. To see this, it is enough to observe that, since the system is linear, the stability of the system can be analyzed from the initial time $t_\kappa$ with $x(t_\kappa)$ playing the role of initial conditions.}

\subsection{Stability under minimum dwell-time - Nominal case}\label{sec:minDT}

Let us consider now the family of piecewise differentiable parameter trajectories given by
\begin{equation}\label{eq:minDTparam}
 \mathscr{P}_{\hspace{-1mm}{\scriptscriptstyle\geqslant\bar{T}}}:=\left\{\begin{array}{c}
    \rho:\mathbb{R}_{\ge0}\mapsto \mathcal{P}\left|\begin{array}{c}
     \dot{\rho}(t)\in\mathcal{Q}(\rho(t)),t\in(t_k,t_{k+1}],T_k\ge\bar T\\
    t_0=0,\rho(t_0)=\rho(t_0^+),k\in\mathbb{Z}_{\ge0}
    \end{array}\right.\end{array}\right\}
\end{equation}
where $T_k:=t_{k+1}-t_k$ and $\bar T>0$.  \bll{Unlike in the constant dwell-time case, the jumps here occur aperiodically with the restriction that the minimum time between two consecutive jumps is given by $\bar T$.} We also make the following assumption on the matrix of the system:
\bll{\begin{assumption}\label{hyp:A}
  The state matrix of the system \eqref{eq:mainsyst} is such that $A(\bar{T}+s,\rho(\bar T+s))=A(\bar T,\rho(\bar T+s))$ for all $s\ge0$ and all $\rho(\bar T+s)\in\mathcal{P}$.
\end{assumption}}
\bll{This assumption arises from the fact that, when one wants to stabilize an impulsive or a switched LTI system using a state-feedback controller in the minimum dwell-time setting, a natural structure for this controller is to be timer-dependent until the value of the dwell-time is reached, and from then on, keep the value of the gain of the controller constant with the timer value locked at $\bar T$; see e.g. \cite{Allerhand:11,Briat:13d,Briat:19:Linf} where the same assumption has been considered. This means that if one wants to consider the stability theorem under minimum dwell-time for stabilization purposes, the considered system needs to satisfy such a property. This assumption is hence considered so that this very property is met by the system. This will become much clearer in Section \ref{sec:stabz:minDT} when the stabilization problem will be addressed.}
We then have the following result:
\begin{theorem}[Minimum dwell-time]\label{th:minDT}
\bll{Assume that the parameter trajectories satisfy Assumption \ref{hyp:param}, that the matrix of the system \eqref{eq:mainsyst} satisfies Assumption \ref{hyp:A}, and let $\bar T\in\mathbb{R}_{>0}$ be given.} Assume further that there exist a matrix-valued function $P:[0, \bar{T}]\times\mathcal{P}\mapsto \mathbb{S}^n$, $P(\bar T,\rho)\in\mathbb{S}^n_{\succ0}$, $\rho\in\mathcal{P}$, and a scalar $\eps>0$ such that the conditions
       \begin{equation}\label{eq:minDT:2}
      \partial_\tau P(\tau,\rho)+\sum_{i=1}^N\partial_{\rho_i} P(\tau,\rho)\mu_i+\He[P(\tau,\rho)A(\tau,\rho)]+\eps I\preceq0,
    \end{equation}
     \begin{equation}\label{eq:minDT:1}
      \sum_{i=1}^N\partial_{\rho_i} P(\bar{T},\rho)\mu_i+\He[P(\bar{T},\rho)A(\bar{T},\rho)]+\eps I\preceq0,
    \end{equation}
    and
     \begin{equation}\label{eq:minDT:3}
      J(\rho)^{\T}P(0,\rho^+)J(\rho)-P(\bar{T},\rho)\preceq0
    \end{equation}
    hold for all $(\tau,\rho,\rho^+,\mu)\in[0, \bar{T}]\times\mathcal{P}\times\mathcal{P}\times \mathcal{D}^v$. Then, the LPV system \eqref{eq:mainsyst} with parameter trajectories in $\mathscr{P}_{\hspace{-1mm}{\scriptscriptstyle\geqslant \bar{T}}}$ is uniformly exponentially stable.

    \bll{Moreover, when the above conditions hold, then there exists a matrix-valued function $Q:\mathcal{P}\mapsto\mathbb{S}^n_{\succ0}$ such that
\begin{equation}
  M^{\T}Q(\rho^+)M-Q(\rho)\prec0
\end{equation}
holds for all $\textstyle M\in\bigcup_{T\ge\bar T}\mathcal{M}(T,\rho)$ and all $(\rho,\rho^+)\in\mathcal{P}\times\mathcal{P}$ and the pre-jump embedded discrete-time system \eqref{eq:mainsystDT1} is uniformly exponentially stable.}\hfill\mendth
\end{theorem}
\bll{\begin{proof}
As the proof is similar to that of Theorem \ref{th:cstDT}, it will be only sketched and details will be only given at places where the proofs differ. Let us consider here the function
\begin{equation}\label{eq:djslkjdsjdjsldjldjsldjlkasj}
V(x,\tau,\rho):=x^{\T}P(\min\{\tau,\bar T\},\rho)x
\end{equation}
where we can see that this function is similar to that of Theorem \ref{th:cstDT} with the difference that the matrix $P$ is locked at $\tau=\bar T$ when the timer reaches this value. The function $V(x,\bar T,\rho)$ is readily seen to be positive definite for the same reasons as in Theorem \ref{th:cstDT}. Pre- and post-multiplying \eqref{eq:minDT:2} by $x(t_k+\tau)^{\T}$ and $x(t_k+\tau)$, letting $(\rho,\mu)\leftarrow(\rho(t_k+\tau),\dot{\rho}(t_k+\tau))$ and integrating from 0 to $\bar{T}$ with respect to $\tau$ yields
\begin{equation}\label{eq:glab1}
V(x(t_k+\bar{T}),\bar{T},\rho(t_k+\bar{T}))\le \exp\left(-\dfrac{\eps}{\alpha_2}\bar T\right)V(x(t_k^+),0,\rho(t_k^+))
\end{equation}
%
Pre- and post-multiply now \eqref{eq:minDT:1} by $x(t)^{\T}$ and $x(t)$, substituting $\rho=\rho(t)$ for $t\in[t_k+\bar{T},t_{k+1}]$ shows that is equivalent to saying that
\begin{equation}
  \dot{V}(x(t),t-t_k,\rho(t))\le -\eps||x(t)||_2^2\le-\dfrac{\eps}{\alpha_2} V(x(t),t-t_k,\rho(t)),t\in(t_k+\bar{T},t_{k+1}].
\end{equation}
Therefore, we have that
\begin{equation}\label{eq:kdl;askd;lska;lk;l;klk;l;kl;kl;kl456465446}
  V(x(t),t-t_k,\rho(t))\le \exp\left(-\dfrac{\eps}{\alpha_2}(t-t_k-\bar{T})\right)V(x(t_k+\bar{T}),\bar{T},\rho(t_k+\bar{T})),t\in[t_k+\bar{T},t_{k+1}],
\end{equation}
and  this must hold for all $x(t_k+\bar{T})\in\mathbb{R}^n$, all $\rho(t_k+\bar{T})\in\mathcal{P}$ and all $\rho$ satisfying the boundedness and continuous differentiability conditions. Combining \eqref{eq:glab1} and \eqref{eq:kdl;askd;lska;lk;l;klk;l;kl;kl;kl456465446} yields
\begin{equation}\label{eq:jdksjkdlsajklddsljksjlk}
\begin{array}{rcl}
  V(x(t),t-t_k,\rho(t))&\le &\exp\left(-\dfrac{\eps}{\alpha_2}(t-t_k)\right)V(x(t_k^+),0,\rho(t_k^+))\\
  &\le &\exp\left(-\dfrac{\eps}{\alpha_2}(t-t_k)\right)V(x(t_k),T_k,\rho(t_k))\\
  &\le &\exp\left(-\dfrac{\eps}{\alpha_2}(t-t_k)\right)V(x(t_k),\bar T,\rho(t_k))
\end{array}
\end{equation}
for all $t\in(t_k,t_{k+1}]$, where the second inequality follows from \eqref{eq:minDT:3} and the last one from the definition of $V$ in \eqref{eq:djslkjdsjdjsldjldjsldjlkasj}. In particular, we have that
\begin{equation}\label{eq:jdkjskdjksjdjskdjskdjsdksjdjsdksj}
  V(x(t_{k+1}),\bar T,\rho(t_{k+1}))\le\xi V(x(t_{k}),\bar T,\rho(t_{k}))
\end{equation}
where $\xi=\exp\left(-\eps\bar T/\alpha_2\right)$. Since $\xi\in(0,1)$, then $||x(t_k)||\to0$ as $k\to\infty$. Using the same argument as in the proof of Theorem \ref{th:cstDT}, we can prove that $||x(t)||\to0$ as $t\to\infty$ and that the impulsive LPV system \eqref{eq:mainsyst} with parameter trajectories in $\mathscr{P}_{\hspace{-1mm}{\scriptscriptstyle\geqslant\bar{T}}}$ is uniformly exponentially stable under minimum dwell-time $\bar T$. Finally, the last statement can be proven by expanding \eqref{eq:jdkjskdjksjdjskdjskdjsdksjdjsdksj} and letting $Q(\cdot)=P(\bar T,\cdot)$.
\end{proof}}

\bll{This result also deserves a few comments. One can see first that it involves one more condition than the constant dwell-time result and that this extra condition characterizes the long-run stability of the system when the dwell-time exceeds the minimum dwell-time value. Another interesting point is that despite the fact that the dwell-time can be arbitrarily large, one just needs to verify the conditions over $[0,\bar T]$. This dramatically simplifies the problem as it allows for the consideration of matrix-valued polynomials $P(\tau,\rho)$ and $A(\tau,\rho)$ that would have grown without bound if we had to check the conditions over $[0,T_k]$ instead of $[0,\bar T]$. Finally, for the same reasons as for Theorem \ref{th:cstDT}, this result remains valid for sequences of dwell-times eventually satisfying a minimum dwell-time condition; i.e. for which there exists a $\kappa\in\mathbb{Z}_{\ge0}$ such that for all $k\ge\kappa$, we have that $T_k\ge\bar T$.}


\subsection{Stability under minimum dwell-time - Uncertain case}\label{sec:minDT2}

It seems interesting to address the uncertain case where the matrices of  the system are subject to constant polytopic uncertainties, that is, with a slight abuse of language, when the matrices of the system are considered to be given by
\begin{equation}\label{eq:uncertain}
  A(\tau,\rho,\lambda)=\sum_{j=1}^M\lambda_jA_j(\tau,\rho)\ \textnormal{ and }  J(\rho,\lambda)=\sum_{j=1}^M\lambda_jJ_j(\rho)
\end{equation}
where $\lambda\in\Lambda_M$ is the uncertain vector which belongs to $M$-unit simplex $\Lambda_M$ defined as
\begin{equation}
  \Lambda_M:=\left\{\lambda\in\mathbb{R}_{\ge0}^M:||\lambda||_1=1\right\}.
\end{equation}
Dealing with such a type of uncertainties can be done using a Lyapunov function that depends affinely on the polytopic parameters. However, due to the product between the Lyapunov matrix $P$ and the matrices of the system $A$ and $J$, the resulting LMI would be quadratic in the polytopic parameter and nonconvex in general. \bll{An elegant way for solving this issue is through the use of so-called dilated LMI conditions which involve slack variables and which can be obtained using Finsler's lemma \cite{SkeltonIG:97a} or the Projection lemma \cite{Gahinet:94a}. Dilated LMI conditions are not novel and have been already considered in the past in various contexts \cite{Geromel:98,Oliveira:99,Tuan:03,Wu:10,Briat:book1,Ebihara:15}. Note also that the polytopic parameter $\lambda$ can be made time-varying at the expense of having to deal with their derivatives using, for instance, the approach described in \cite{Briat:book1}. This is omitted for brevity.} As in the previous section, we consider the following assumption:
\bll{\begin{assumption}\label{hyp:Al}
  The state matrix of the uncertain system \eqref{eq:mainsyst}-\eqref{eq:uncertain} is such that $A(\bar{T}+s,\rho(\bar T+s),\lambda)=A(\bar T,\rho(\bar T+s),\lambda)$ for all $s\ge0$ and all $\lambda\in\Lambda_M$ and all $\rho(\bar T+s)\in\mathcal{P}$.
\end{assumption}}
Then, we have the following result:
\begin{theorem}[Minimum dwell-time]\label{th:minDTrelax}
\bll{Assume that the parameter trajectories satisfy Assumption \ref{hyp:param}, that the matrix of the uncertain system \eqref{eq:mainsyst}-\eqref{eq:uncertain}  satisfies Assumption \ref{hyp:Al}, and let $\bar T\in\mathbb{R}_{>0}$ be given.} Assume further that there exist matrix-valued functions $P_j:[0, \bar{T}]\times\mathcal{P}\mapsto \mathbb{S}^n$, $P_j(0,\rho)\in\mathbb{S}^n_{\succ0}$, $j=1,\ldots,M$, $\rho\in\mathcal{P}$, $X_1,X_2:[0, \bar{T}]\times\mathcal{P}\times\mathcal{D}^v\mapsto \mathbb{R}^{n\times n}$, $Z_1,Z_2:\mathcal{P}\times\mathcal{P}\mapsto \mathbb{R}^{n\times n}$, and a scalar $\eps>0$ such that the conditions
       \begin{equation}\label{eq:minDT:2relax}
       \begin{bmatrix}
         -\He[X_1(\tau,\rho,\mu)] & X_1(\tau,\rho,\mu)^{\T}A_j(\tau,\rho)-X_2(\tau,\rho,\mu)+P_j(\tau,\rho)\\
         \star & \partial_\tau \bll{P_j(\tau,\rho)}+\sum_{i=1}^N\partial_{\rho_i} P_j(\tau,\rho)\mu_i+\He[X_2(\tau,\rho,\mu)^{\T}A_j(\tau,\rho)]+\eps I
       \end{bmatrix}\preceq0
    \end{equation}
     \begin{equation}\label{eq:minDT:1relax}
     \begin{bmatrix}
         -\He[X_1(\bar{T},\rho,\mu)] & X_1(\bar{T},\rho,\mu)^{\T}A_j(\bar{T},\rho)-X_2(\bar{T},\rho,\mu)+P_j(\bar{R},\rho)\\
         \star & \sum_{i=1}^N\partial_{\rho_i} P_j(\bar{T},\rho)\mu_i+\He[X_2(\bar{T},\rho,\mu)^{\T}A_j(\bar{T},\rho)]+\eps I
       \end{bmatrix}\preceq0
    \end{equation}
    and
     \begin{equation}\label{eq:minDT:3relax}
     \begin{bmatrix}
       P_j(0,\rho^+)-\He[Z_1(\rho^+,\rho)] & Z_1(\rho^+,\rho)^{\T}J_j(\rho)-Z_2(\rho^+,\rho)\\
        & -P_j(\bar{T},\rho)+\He[Z_2(\rho^+,\rho)^{\T}J_j(\rho)]
     \end{bmatrix}\preceq0
       \end{equation}
      hold for all $(\tau,\rho,\rho^+,\mu)\in[0, \bar{T}]\times\mathcal{P}\times\mathcal{P}\times \mathcal{D}^v$ and all $j=1,\ldots,M$. Then, the uncertain LPV system \eqref{eq:mainsyst}-\eqref{eq:uncertain} with parameter trajectories in $\mathscr{P}_{\hspace{-1mm}{\scriptscriptstyle\geqslant \bar{T}}}$ is uniformly exponentially stable.

     \bll{Moreover, there exists a matrix-valued function $Q:\mathcal{P}\times\Lambda_M\mapsto\mathbb{S}_{\succ0}^n$ such that the condition
      \begin{equation}
        M^{\T}Q(\rho^+,\lambda)M-Q(\rho,\lambda)\prec0
      \end{equation}
      holds for all $M\in\bigcup_{T\ge\bar T}\mathcal{M}_r(T,\rho,\lambda)$, all $(\rho,\rho^+)\in\mathcal{P}\times\mathcal{P}$ and all $\lambda\in\Lambda_M$ where
      \begin{equation}
        \mathcal{M}_r(T,\rho,\lambda):=\left\{\Psi(T)J_k(\varpi,\lambda)\left|\begin{array}{l}
    \dot{\Psi}(\tau)=A(\tau,\varrho(\tau),\lambda)\Psi(\tau),\Psi(0)=I,\\
    \varrho(\tau)\in\mathcal{P}, \dot{\varrho}(\tau)\in\mathcal{Q}(\varrho(\tau)),\tau\in[0,T]
\end{array}\right.\right\}
      \end{equation}
      where $J_0(\varpi,\lambda)=I$ and $J_k(\varpi,\lambda)=J(\varpi,\lambda)$, $k\ge1$. As a result, the pre-jump embedded discrete-time associated with \eqref{eq:mainsyst}-\eqref{eq:uncertain} is robustly uniformly exponentially stable.}
\end{theorem}
\begin{proof}
  Multiplying the conditions by $\lambda_j$ and summing over $j=1,\ldots,M$ yields the same conditions as in Theorem \ref{th:minDTrelax} with $(P_j(\tau,\rho),A_j(\tau,\rho),J_j(\rho))$ replaced by $(P(\tau,\rho,\lambda),A(\tau,\rho,\lambda),J(\rho,\lambda))$ where $\textstyle P(\tau,\rho,\lambda)=\sum_{j=1}^M\lambda_jP_j(\tau,\rho)$, $\textstyle A(\tau,\rho,\lambda)=\sum_{j=1}^M\lambda_jA_j(\tau,\rho)$ and $\textstyle J(\rho,\lambda)=\sum_{j=1}^M\lambda_jJ_j(\rho)$. The rest of the proof is based on a direct application of Finsler's lemma \cite{SkeltonIG:97a}. Indeed, the condition \eqref{eq:minDT:2relax} can be reformulated as
  \begin{equation}
  \begin{bmatrix}
    0 & P(\tau,\rho,\lambda)\\
    P(\tau,\rho,\lambda) & \partial_\tau P(\tau,\rho,\lambda)+\sum_{i=1}^N\partial_{\rho_i} P(\tau,\rho,\lambda)\mu_i
  \end{bmatrix}+\He\left(\begin{bmatrix}
    X_1(\tau,\rho,\mu)^{\T}\\
   X_2(\tau,\rho,\mu)^{\T}
  \end{bmatrix}\begin{bmatrix}
    -I & A(\tau,\rho)
  \end{bmatrix}\right)\preceq0.
  \end{equation}
  By virtue of Finsler's Lemma, this inequality is equivalent to
  \begin{equation}
  \begin{bmatrix}
    A(\tau,\rho,\lambda)\\
    I
  \end{bmatrix}^{\T}\begin{bmatrix}
    0 & P(\tau,\rho,\lambda)\\
    P(\tau,\rho,\lambda) & \partial_\tau P(\tau,\rho,\lambda)+\sum_{i=1}^N\partial_{\rho_i} P(\tau,\rho,\lambda)\mu_i
  \end{bmatrix}\begin{bmatrix}
    A(\tau,\rho,\lambda)\\
    I
  \end{bmatrix}\preceq0
  \end{equation}
  which is readily shown to be identical to \eqref{eq:minDT:2} with $A(\tau,\rho)$ and $P(\tau,\rho)$ replaced by $A(\tau,\rho,\lambda)$ and $P(\tau,\rho,\lambda)$, respectively. The equivalence of the other conditions with \eqref{eq:minDT:1} and \eqref{eq:minDT:3} are proved in the exact same way. \bll{The adaptation of the concluding statement also follows from the straightforward generalization of the set $\mathcal{M}(T,\varpi)$ to the uncertain case.}
\end{proof}

Interestingly, the decoupling between the Lyapunov matrix $P$ and the matrices of the system $A$ and $J$ is not only useful for efficiently dealing with polytopic uncertainties but can also be used to enforce certain constraints on the controller when stabilization is aimed. The explanation is that when structural constraints are enforced on the controller, the problem can only be kept convex by enforcing a similar structure on the Lyapunov function, thereby making the approach more conservative. However, when using the relaxed conditions, the structural constraints are delegated to the slack-variables $X_1$ and $X_2$ while the Lyapunov function is left unchanged -- a procedure that has been proven to dramatically limit the increase of the conservatism. This will be further detailed in Section \ref{sec:stabz1}.

\subsection{Connection with switched impulsive LPV systems}\label{sec:switched}

If we extend the original system by adding a piecewise constant parameter $\sigma$ which takes values in $\{1,\ldots,M\}$, we obtain the following class of systems taking the form of a switched impulsive LPV system with piecewise differentiable parameters:
\begin{equation}\label{eq:mainsystH_RDTsw}
\begin{array}{rcl}
    \dot{x}(t)&=&A(t-t_k,\rho(t),\bll{\sigma(t)})x(t),\ t\in(t_k,t_{k+1}],k\ge0\\
    x(t_k^+)&=&J(\rho(t_k),\sigma(t_k^+),\sigma(t_k))x(t_k),\ k\in\mathbb{Z}_{>0}.
  \end{array}
\end{equation}
Note that the jump map now depends on both the current value of the discrete-valued parameter $\sigma$ and the next one $\sigma^+$. While this may look non-causal at first sight, this is actually not the case since one can decide first the next value for the parameters $\rho$ and $\sigma$ and then make the state jump accordingly. Whenever, $J=I_n$ we recover a standard LPV switched system. We also make the following assumption:
\bll{\begin{assumption}\label{hyp:Ai}
  The state matrix of the switched impulsive LPV system \eqref{eq:mainsyst}-\eqref{eq:uncertain} is such that $A(\bar{T}+s,\rho(\bar T+s),\sigma(\bar T+s))=A(\bar T,\rho(\bar T+s),\sigma(\bar T+s))$ for all $s\ge0$, all $\sigma\in\{1,\ldots,M\}$ and all $\rho(\bar T+s)\in\mathcal{P}$.
\end{assumption}}

This leads to the following result which an adaptation of Theorem \ref{th:minDT} to the switched impulsive LPV systems case:
\begin{corollary}[Minimum dwell-time - Switched LPV systems]\label{th:minDTsw}
\bll{Assume that the parameter trajectories satisfy Assumption \ref{hyp:param}, that the matrix of the impulsive switched LPV system \eqref{eq:mainsystH_RDTsw} satisfies Assumption \ref{hyp:Al}, and let $\bar T\in\mathbb{R}_{>0}$ be given.}  Assume further that there exist a matrix-valued function $P:[0, \bar{T}]\times\mathcal{P}\times\{1,\ldots,M\}\mapsto \mathbb{S}^n$, $P(0,\rho,\sigma)\in\mathbb{S}^n_{\succ0}$, $(\rho,\sigma)\in\mathcal{P}\times\{1,\ldots,M\}$, and a scalar $\eps>0$ such that the conditions
       \begin{equation}\label{eq:minDT:2:relax}
      \partial_\tau P(\tau,\rho,\sigma)+\sum_{i=1}^N\partial_{\rho_i} P(\tau,\rho,\sigma)\mu_i+\He[P(\tau,\rho,\sigma)A(\tau,\rho,\sigma)]+\eps I\preceq0,
    \end{equation}
     \begin{equation}\label{eq:minDT:1:relax}
      \sum_{i=1}^N\partial_{\rho_i} P(\bar{T},\rho,\sigma)\mu_i+\He[P(\bar{T},\rho,\sigma)A(\bar{T},\rho,\sigma)]+\eps I\preceq0
    \end{equation}
    and
     \begin{equation}\label{eq:minDT:3:relax}
      J(\rho,\sigma^+,\sigma)^{\T}P(0,\rho^+,\sigma^+)J(\rho,\sigma^+,\sigma)-P(\bar{T},\rho,\sigma)\preceq0
    \end{equation}
    hold for all $\sigma,\sigma^+\in\{1,\ldots,M\}$, $\sigma^+\ne \sigma$, and all  $(\tau,\rho,\rho^+,\mu)\in[0, \bar{T}]\times\mathcal{P}\times\mathcal{P}\times \mathcal{D}^v$. Then, the switched impulsive LPV system \eqref{eq:mainsystH_RDTsw} with parameter trajectories in $\mathscr{P}_{\hspace{-1mm}{\scriptscriptstyle\geqslant \bar{T}}}$ is uniformly exponentially stable.

     \bll{Moreover, there exists a matrix-valued function $Q:\mathcal{P}\times\{1,\ldots,M\}\mapsto\mathbb{S}_{\succ0}^n$ such that the condition
      \begin{equation}
        M^{\T}Q(\rho^+,\sigma^+)M-Q(\rho,\sigma)\prec0
      \end{equation}
      holds for all $M\in\bigcup_{T\ge\bar T}\mathcal{M}_s(T,\rho,\sigma)$, all $(\rho,\rho^+)\in\mathcal{P}\times\mathcal{P}$ and all $(\sigma,\sigma^+)\in\{1,\ldots,M\}\times\{1,\ldots,M\}$ where
      \begin{equation}
        \mathcal{M}_s(T,\varpi,\sigma):=\left\{\Psi(T)J_k(\varpi,\sigma^+,\sigma)\left|\begin{array}{l}
        \dot{\Psi}(\tau)=A(\tau,\varrho(\tau),\sigma^+)\Psi(\tau),\Psi(0)=I,\\
        \varrho(\tau)\in\mathcal{P}, \dot{\varrho}(\tau)\in\mathcal{Q}(\varrho(\tau)),\tau\in[0,T]\\
        \sigma^+\in\{1,\ldots,M\}
\end{array}\right.\right\}
      \end{equation}
      where $J_0(\varpi,\sigma^+,\sigma)=I$ and $J_k(\varpi,\sigma^+,\sigma)=J(\varpi,\sigma^+,\sigma)$, $k\ge1$. As a result, the pre-jump embedded discrete-time associated with \eqref{eq:mainsystH_RDTsw} is uniformly exponentially stable.}
\end{corollary}

\subsection{Connection with quadratic and robust stability}\label{sec:QR}

Following the same lines as in \cite{Briat:15d}, it can be shown that the minimum dwell-time result stated in Theorem \ref{th:minDT} naturally generalizes and unifies the quadratic and robust stability conditions for LPV systems; i.e. whenever $J(\rho)=I_n$ in \eqref{eq:mainsyst}.  We recall first the notions of quadratic and robust stability:
\begin{define}\label{def:stab}
  \bll{Assume that the parameter trajectories satisfy Assumption \ref{hyp:param} and that the matrices of the LPV system \eqref{eq:mainsyst} satisfy $A(\tau,\rho)\equiv A(\rho)$ (timer-independent matrix) and $J(\rho)=I_n$. Then, the LPV system \eqref{eq:mainsyst} is }
  \begin{enumerate}[(a)]
    \item \textbf{quadratically stable} if there exists a matrix $P_q\in\mathbb{S}_{\succ0}^n$ such that
    \begin{equation}\label{eq:quadstab}
        A(\rho)^{\T}P_q+P_qA(\rho)\prec0
  \end{equation}
  holds for all $\rho\in\mathcal{P}$.
  \item \textbf{robustly stable} if there exists  a continuously differentiable matrix-valued function $P_r:\mathcal{P}\mapsto\mathcal{S}_{\succ0}^n$ such that
  \begin{equation}\label{eq:robstab}
      \sum_{i=1}^N\partial_{\rho_i} P_r(\rho)\mu_i+A(\rho)^{\T}P_r(\rho)+P_r(\rho)A(\rho)\prec0
  \end{equation}
  holds for all $\rho\in\mathcal{P}$ and all $\mu\in\mathcal{D}^v $.
  \end{enumerate}
\end{define}

We then have the following result:
\begin{theorem}
\bll{Assume that the parameter trajectories satisfy Assumption \ref{hyp:param} and that the matrices of the LPV system \eqref{eq:mainsyst} satisfy $A(\tau,\rho)\equiv A(\rho)$ (timer-independent matrix) and $J(\rho)=I_n$. Assume further that the parameter trajectories belong to $\mathscr{P}_{\hspace{-1mm}{\scriptscriptstyle\geqslant\bar{T}}}$ in \eqref{eq:minDTparam}}. Then, the following statements hold:
 \begin{enumerate}[(a)]
   \item  When $\bar T=0$ and $t_k\to\infty$ as $k\to\infty$, then the conditions of Theorem \ref{th:minDT} reduce the quadratic stability condition in Definition \ref{def:stab}.
   \item When $\bar T=\infty$, then the conditions of Theorem \ref{th:minDT} reduce to robust stability condition in Definition \ref{def:stab}.
 \end{enumerate}
\end{theorem}
\begin{proof}
  We prove first statement (a).  First note that since $\bar T=0$ and $t_k\to\infty$ as $k\to\infty$, then there is no accumulation point in the sequence of jumping instants and, therefore, the solution of the \bll{impulsive LPV  system is complete; i.e. it is defined for all $t\ge0$}. Clearly, if $\bar T=0$, then we will have that $P(\bar T,\rho)=P(0,\rho)$ and $P(\tau,\rho)=P(\bar T,\rho)$ for all $\tau\ge\bar T$. Hence, we obtain that $P(\tau,\rho)$ is constant for each $\rho$ and equal to $P(0,\rho)$. Substituting that expression in the jump condition \bll{\eqref{eq:minDT:3}} implies that the two following inequalities
\begin{equation}
  \begin{array}{rcl}
    P(0,\theta)-P(\tau,\eta)&=&P(0,\theta)-P(0,\eta)\preceq0,\\
    P(0,\eta)-P(\tau,\theta)&=&P(0,\eta)-P(0,\theta)\preceq0
  \end{array}
\end{equation}
hold for any $\theta\ne\eta$. The first expression arises when $\rho^+=\theta$ and $\rho=\eta$ while the other expression is when $\rho^+=\eta$ and $\rho=\theta$. Therefore, for those expressions above to be satisfied, it is necessary that we have $P(0,\eta)-P(0,\theta)=0$ for all $\eta\ne\theta\in\mathcal{P}$. This can only be true if $P$ is actually independent of the parameter and hence we need to have $P(\tau,\theta)=P_q$ for some $P_q\succ0$. Substituting that in \eqref{eq:minDT:1}-\eqref{eq:minDT:2} yield the quadratic stability condition \eqref{eq:quadstab}.

To prove statement (b), it is enough to remark that when $\bar T=\infty$, the system does not exhibit any impulsive behavior anymore and reduces to a continuous-time system. Dropping then the timer dependence i.e. (letting $P(\tau,\theta)=P_r(\theta)$) and ignoring the condition \eqref{eq:minDT:3} yield the robust stability condition as the conditions \eqref{eq:minDT:1}-\eqref{eq:minDT:2} both reduce to \eqref{eq:robstab} in this case.
\end{proof}

\subsection{Computational considerations}\label{sec:computational}

\begin{mybox*}
\caption{SOS program associated with Theorem \ref{th:minDT}}\label{box}
{\vspace{1mm}}
\noindent\fbox{
\parbox{\textwidth}{
    Find polynomial matrices $S,\Gamma_i,\Omega_i:[0,\bar{T}]\times\mathcal{P}\mapsto\mathbb{S}^n$, $\Theta_i,\Xi_i:\mathcal{P}\times\mathcal{P}\mapsto\mathbb{S}^n$, $\Upsilon_0,\Upsilon_i: [0,\bar{T}]\times\mathcal{P}\times \mathcal{D}^v \mapsto\mathbb{S}^n$, $i=1,\ldots,M$, such that
      \begin{itemize}
        \item $\Gamma_0,\Gamma_i,\Theta_i,\Xi_i,\Upsilon_i,\Upsilon_0,\Omega_i$, $i=1,\ldots,M$, are SOS matrices for all $\mu\in\mathcal{D}^v$,
        %
        \item \bll{$P(\bar T,\rho)-\sum_{i=1}^{M}\Gamma_i(\rho)g_i(\rho)-\eps I_n$ is an SOS matrix},
        \item $-\sum_{i=1}^N\partial_{\rho_i} P(\tau,\rho)\mu_i-\partial_\tau P(\tau,\rho)-\He[P(\tau,\rho)A(\rho)] -\sum_{i=1}^{M}\Upsilon_i(\tau,\rho,\mu)g_i(\rho) -\Upsilon_0(\tau,\rho,\mu)\tau( \bar T-\tau)-\eps I_n$ is an SOS matrix for all $\mu\in\mathcal{D}^v$,
            \item $ -\sum_{i=1}^N\partial_{\rho_i} P(\bar{T},\rho)\mu_i-\He[P(\bar{T},\rho)A(\rho)]-\sum_{i=1}^{M}\Omega_i(\rho,\mu)g_i(\rho)-\eps I_n$  is an SOS matrix for all $\mu\in\mathcal{D}^v$,
        \item ${P(\bar T,\rho)-J(\rho)^{\T}P(0,\rho^+)J(\rho)-\sum_{i=1}^{N}\Theta_i(\rho^+,\rho)g_i(\rho^+)} -\sum_{i=1}^{M}\Xi_i(\rho^+,\rho)g_i(\rho)$ is an SOS matri.x
    \end{itemize}}}
\end{mybox*}

The conditions formulated in Theorem \ref{th:minDT} are infinite-dimensional semidefinite programs \blue{which can not be solved directly}. To make them tractable, we propose to consider an approach based on sum of squares programming \citep{Parrilo:00} that will result in a finite-dimensional semidefinite program which can then be solved using standard solvers such as SeDuMi \citep{Sturm:01a}. The conversion to a semidefinite program can be automatically performed using the package SOSTOOLS \citep{sostools3} to which we input the SOS program corresponding to the considered conditions. We illustrate below how an SOS program associated with some given conditions can be obtained. The set $\mathcal{P}$ defined in Assumption \ref{hyp:param} can be implicitly described as
\begin{equation}
  \mathcal{P}=\left\{\rho\in\mathbb{R}^N: g_{i}(\rho)\ge0, i=1,\ldots,N\right\}
\end{equation}
where $g_i(\rho):=(\bar{\rho}_i-\rho_i)(\rho_i-\underline{\rho}_i)$, $i=1,\ldots,N$. Additionally, we have that
\begin{equation}
  [0, \bar{T}]=\left\{\tau\in\mathbb{R}:\ f(\tau):=\tau( \bar{T}-\tau)\ge0\right\}.
\end{equation}
In what follows, we say that a symmetric polynomial matrix $\Theta(\cdot)$ is a sum of squares matrix (SOS matrix) or is SOS, for simplicity, if there exists a polynomial matrix $\Xi(\cdot)$ such that $\Theta(\cdot)=\Xi(\cdot)^{T}\Xi(\cdot)$. The following result provides the SOS formulation of Theorem \ref{th:minDT}:
\begin{proposition}\label{prog:periodic}
  Let $\eps, \bar{T}>0$ be given and  assume that the sum of squares program in Box \ref{box} is feasible. Then, the conditions of Theorem \ref{th:minDT} hold with the computed polynomial matrix $P(\tau,\theta)$ and the system \eqref{eq:mainsyst} is asymptotically stable for all $\rho\in\mathscr{P}_{\hspace{-1mm}{\scriptscriptstyle\geqslant\bar{T}}}$.
\end{proposition}

\begin{remark}
  When the parameter set $\mathcal{P}$ is also defined by equality constraints $h_i(\theta)=0$, $i=1,\ldots, M'$, these constraints can be simply added in the sum of squares programs in the same way as the inequality constraints, but with the particularity that the corresponding multiplier matrices be simply symmetric instead of being SOS matrices. \bll{When the set of values for the derivative is not polyhedral, then the conditions need to be checked for all $\mu=\dot\rho\in\mathcal{Q}(\rho)$ or an approximation of it.}
\end{remark}

\subsection{Examples}\label{sec:examples}

\blue{We consider now two examples. The first one is a 2-dimensional toy example considered in \citep{Xie:97} whereas the second one is a 4-dimensional system considered in \citep{Wu:95} and inspired from an automatic flight control design problem. The numerical calculations have been performed using the package SOSTOOLS \citep{sostools3} and the semidefinite solver SeDuMi \citep{Sturm:01a} on a PC equipped with 12GB of RAM and a processor Intel i7-950 @ 3.07Ghz.}

\begin{example}
  Let us consider here the system \eqref{eq:mainsyst} with the matrices  $J(\rho)=I_n$ and \citep{Xie:97,Briat:15d}
\begin{equation}\label{eq:ex1}
  A(\rho)=\begin{bmatrix}
    0 & &1\\
    -2-\rho &\ &-1
  \end{bmatrix}
\end{equation}
where the time-varying parameter $\rho(t)$ takes values in $\mathcal{P}=[0,\bar{\rho}]$, $\bar{\rho}>0$. It is known \citep{Xie:97} that this system is quadratically stable if and only if $\bar{\rho}\le3.828$ but it is was later proven in the context of piecewise constant parameters \citep{Briat:15d} that this bound can be improved provided that discontinuities do not occur too often. We now apply the conditions of Theorem \ref{th:minDT} in order to characterize the impact of parameter variations between discontinuities. To this aim, we consider that $|\dot{\rho}(t)|\le\nu$ with $\nu\ge0$ and that $\bar{\rho}\in\{0,0.1,\ldots,10\}$. For each value for the parameter upper-bound $\bar{\rho}$ in that set, we solve for the conditions of Theorem \ref{th:minDT} to get estimates (i.e. upper-bounds) for the minimum dwell-times. We use here $\eps=0.01$ and polynomials of degree 4 in the sum of squares programs. Note that we have, in this case, $M=1$, $M'=0$ and $g_1(\rho)=\rho(\bar{\rho}-\rho)$. The complexity of the approach can be evaluated here through the number of primal and dual variables of the semidefinite program which are 2409 and 315, respectively. The average preprocessing and solving times are  given by 6.04sec and 1.25sec, respectively. The results are Fig.~\ref{fig:2} where we can see that the obtained minimum values for the dwell-times increase with the rate of variation $\nu$ of the parameter, which is an indicator of the fact that increasing the rate of variation of the parameter tends to destabilize the system and, consequently, the dwell-time needs to be increased in order to preserve the overall stability of the system.

\begin{figure}[h]
  \centering
  \includegraphics[width=0.8\textwidth]{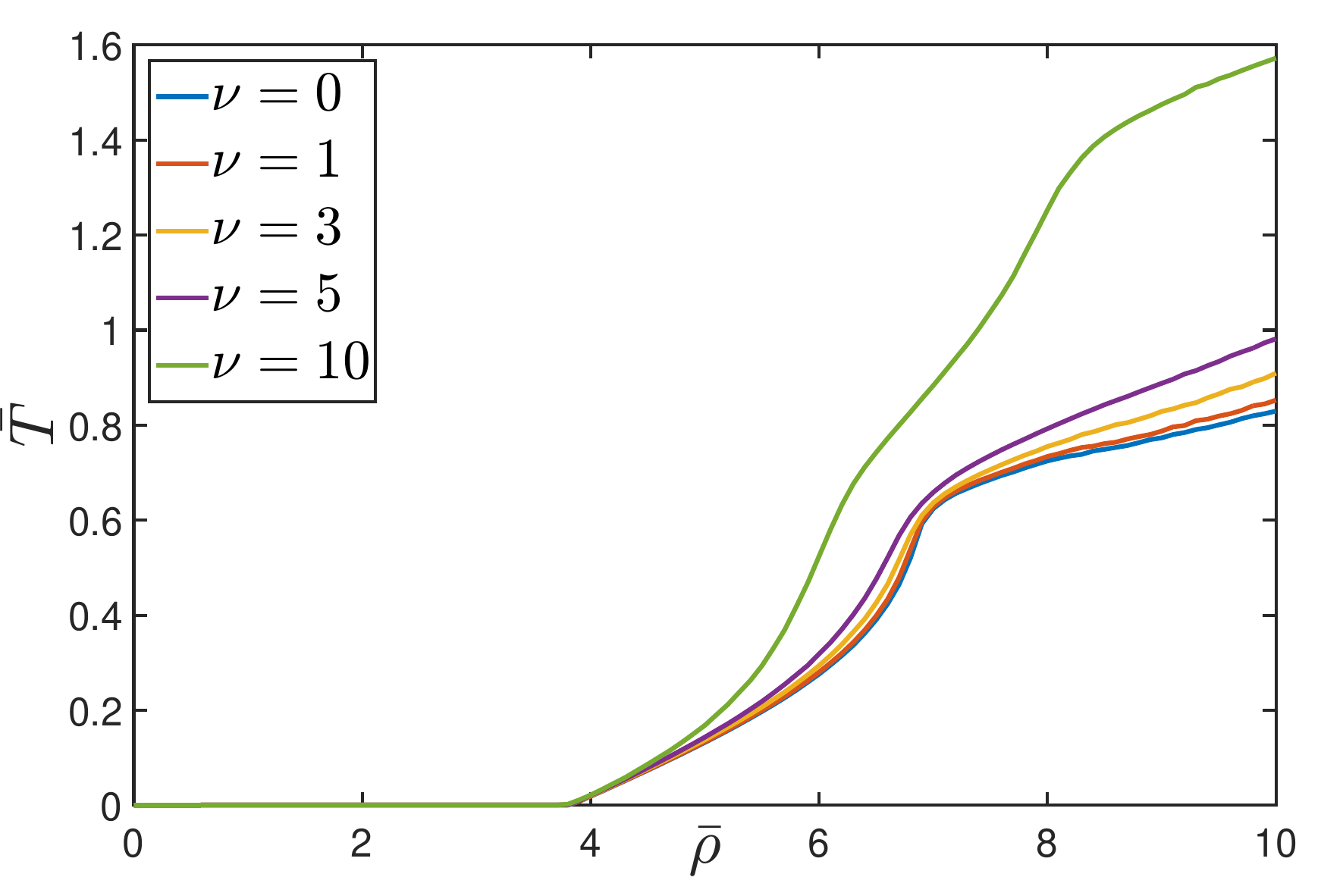}
  \caption{Evolution of the computed minimum upper-bound on the minimum stability-preserving minimum dwell-time using Theorem \ref{th:minDT} for the system \eqref{eq:mainsyst}-\eqref{eq:ex1} with $|\dot{\rho}|\le\nu$ using an SOS approach with polynomials of degree 4.}\label{fig:2}
\end{figure}
\end{example}


\begin{example}
  Let us consider now the system \eqref{eq:mainsyst} with the matrices $J(\rho)=I_n$ and \citep[p. 55]{Wu:95}:
  \begin{equation}\label{eq:ex2}
  A(\rho)=\begin{bmatrix}
   3/4 &\ & 2 &\ & \rho_1 &\ & \rho_2\\
    0 &\ & 1/2 &\ & -\rho_2 &\ & \rho_1\\
  -3\upsilon\rho_1/4   &\ & \upsilon\left(\rho_2-2\rho_1\right) &\ & -\upsilon &\ & 0\\
  -3\upsilon\rho_2/4    &\ & \upsilon\left(\rho_1-2\rho_2\right)  &\ & 0 &\ & -\upsilon
      \end{bmatrix}
\end{equation}
where $\upsilon=15/4$ and $\rho\in\mathcal{P}=\{z\in\mathbb{R}^2:||z||_2=1\}$. It has been shown in \citep{Wu:95} that this system is not quadratically stable but was proven to be stable under minimum dwell-time equal to 1.7605 when the parameter trajectories are piecewise constant \citep{Briat:15d}. We propose now to quantify the effects of smooth parameter variations between discontinuities. \bll{Note, however, that the set $\mathcal{P}$ is a circle, and hence Assumption \ref{hyp:param} does not hold. As mentioned multiple times in the main text, this is not an issue as SOS programming can easily handle such a case since a circle is described by a single polynomial equation. The difficulty lies more at the level of the parameter derivatives and, to resolve this, let us define the parametrization $\rho_1(t)=\cos(\beta(t))$ and $\rho_2(t)=\sin(\beta(t))$ where $\beta(t)$ is piecewise differentiable. Differentiating these equalities yields $\dot{\rho}_1(t)=-\dot{\beta}(t)\rho_2(t)$ and $\dot{\rho}_2(t)=\dot{\beta}(t)\rho_1(t)$ where $\dot{\beta}(t)\in[-\nu,\nu]$, $\nu\ge0$, at all times where $\beta(t)$ is differentiable. Therefore, we can substitute those exact expressions in the conditions of the theorems to obtain conditions that depend on the parameters, the timer variables and the term $\dot{\beta}$ which enters linearly and can be considered as an uncertainty $\dot{\beta}\in[-\nu,\nu]$. However, since this term enters linearly in the conditions, one just needs to check the conditions at the vertices of this interval; i.e. $\dot{\beta}\in\{-\nu,\nu\}$. Note that, in this case, we have $M=0$, $M'=1$ and $h_1(\rho)=\rho_1^2+\rho_2^2-1$ in the SOS program. It is worth mentioning that this approach dramatically simplifies the problem since we could have considered directly the derivatives of the parameters to belong to a set defined by the equality $\dot{\rho}_1(t)^2+\dot{\rho}_2(t)^2=\dot{\beta}(t)^2$ but this approach would have been way more conservative since it does not capture the explicit relationship between the parameters and their derivative.}

\begin{table*}
  \centering
    \caption{Evolution of the computed minimum upper-bound on the minimum dwell-time using Theorem \ref{th:minDT} for the system \eqref{eq:mainsyst}-\eqref{eq:ex2} with $|\dot{\beta}|\le\nu$ using an SOS approach with polynomials of degree $d$. The number of primal/dual variables of the semidefinite program and the preprocessing/solving time are also given.}\label{tab}
  \begin{tabular}{|c||c|c|c|c|c||c|c|c|}
  \hline
   & $\nu=0$ & $\nu=0.1$ & $\nu=0.3$ & $\nu=0.5$ & $\nu=0.8$ & $\nu=0.9$ & primal/dual vars. & time (sec)\\
  \hline
  \hline
    $d=2$  & 2.7282 & 2.9494 & 3.5578 & 4.6317 & 11.6859 & 26.1883 & 9820/1850 & 20/27\\
    $d=4$  & 1.7605 & 1.8881 & 2.2561 & 2.9466 & 6.4539 & \red{num. err.} & 43300/4620 & 212/935\\
    \hline
  \end{tabular}
\end{table*}
We now consider the conditions of Theorem \ref{th:minDT} and we get the results gathered in Table \ref{tab} where we can see that, as expected, when $\nu$ increases then the minimum dwell-time has to increase to preserve stability. Using polynomials of higher degree allows to improve the numerical results at the expense of an increase of the computational complexity. As a final comment, it seems important to point out the failure of the semidefinite solver due to too important numerical errors when $d=4$ and $\nu=0.9$.
\end{example}

\section{Stabilization using hybrid state-feedback LPV controllers}\label{sec:stabz1}

We consider in this section the following extension for the system \eqref{eq:mainsyst}
\begin{equation}\label{eq:mainsyst2}
\begin{array}{rcl}
    \dot{x}(t)&=&A(t-t_k,\rho(t))x(t)+B_c(t-t_k,\rho(t))u_c(t),\ t\in(t_k,t_{k+1}],k\in\mathbb{Z}_{\ge0}\\[0.1em]
    x(t_k^+)&=&J(\rho(t_k))x(t_k)+B_d(\rho(t_k))u_d(k),\ k\in\mathbb{Z}_{>0}\\[0.1em]
    x(t_0^+)&=&x(t_0)=x_0,\ t_0=0
\end{array}
\end{equation}
where $x,x_0\in\mathbb{R}^n$, $u_c\in\mathbb{R}^{m_c}$ and $u_d\in\mathbb{R}^{m_d}$ are the state of the system, the initial condition, the continuous-time control input and the discrete-time control input, respectively. The same assumptions as for the system \eqref{eq:mainsyst} are made for the above system that is:
\bll{\begin{assumption}\label{hyp:AB}
  The matrices of the system \eqref{eq:mainsyst2} are such that $A(\bar{T}+s,\rho(\bar T+s))=A(\bar T,\rho(\bar T+s))$ and $B_c(\bar{T}+s,\rho(\bar T+s))=B_c(\bar T,\rho(\bar T+s))$  for all $s\ge0$ and for all $\rho(\bar T+s)\in\mathcal{P}$.
\end{assumption}}

\subsection{Stabilization by timer-dependent controllers}\label{sec:stabz:minDT}

We consider in this section the following class of state-feedback controllers
\begin{equation}\label{eq:K_MDT}
\begin{array}{rcl}
  u_c(t_k+\tau)&=&\left\{\begin{array}{ll}
    K_c(\tau,\rho(t_k+\tau))x(t_k+\tau),& \tau\in[0,\bar T],\\[0.1em]
    K_c(\bar T,\rho(t_k+\tau))x(t_k+\tau),& \tau\in(\bar T,T_k]
    \end{array}\right.,\\[1em]
  u_d(k)&=&K_d(\rho(t_k))x(t_k)
\end{array}
\end{equation}
where $K_c(\cdot,\cdot)\in\mathbb{R}^{m_c\times n}$ and $K_d(\cdot)\in\mathbb{R}^{m_d\times n}$ are the gains of the controller we would like to determine. The motivation for considering this structure stems from the fact that it mimics the structure of the conditions in Theorem \ref{th:minDT} where the condition \eqref{eq:minDT:2} defined over $[0,\bar T]$ is timer-dependent and the condition \eqref{eq:minDT:1} for dwell-times $\tau\ge\bar T$ is timer-independent for which the value is locked to $\bar T$. Such a structure has been considered before, for instance, in \cite{Allerhand:11,Briat:13d,Briat:14f,Briat:15i} in the context of switched and (stochastic) impulsive systems, respectively.

As the constant dwell-time case can be easily recovered by dropping the timer-independent condition in the minimum dwell-time conditions, only the latter one is addressed for conciseness:
\begin{theorem}[Minimum dwell-time]\label{th:minDT_CT}
\bll{Assume that the matrices of the system \eqref{eq:mainsyst2} satisfy Assumption \ref{hyp:AB}, and let $\bar T\in\mathbb{R}_{>0}$ be given}. Assume further that there exist a continuously differentiable matrix-valued function $X:[0, \bar{T}]\times\mathcal{P}\mapsto \mathbb{S}^n$, $X(\bar T,\rho)\in\mathbb{S}^n_{\succ0}$,  $\rho\in\mathcal{P}$, a matrix-valued function $U_c:[0, \bar{T}]\times\mathcal{P}\mapsto \mathbb{R}^{m_c\times n}$,  a matrix-valued function $U_d:\mathcal{P}\mapsto \mathbb{R}^{m_d\times n}$, and a scalar $\eps>0$ such that the conditions
     \begin{equation}\label{eq:stabzCT1}
       -\sum_{i=1}^N\partial_{\rho_i} X(\bar T,\rho)\mu_i+\He[A(\bar{T},\rho)X(\bar{T},\rho)+B_c(\bar{T},\rho)U_c(\bar{T},\rho)]+\eps I_n\preceq0,
    \end{equation}
       \begin{equation}\label{eq:stabzCT2}
        -\partial_\tau X(\tau,\rho)-\sum_{i=1}^N\partial_{\rho_i} X(\tau,\rho)\mu_i+\He[A(\tau, \rho)X(\tau,\rho)+B_c(\tau,\rho)U_c(\tau,\rho)]+\eps I_n\preceq0
        \end{equation}
    and
     \begin{equation}\label{eq:stabzCT3}
     \begin{bmatrix}
       -X(0,\rho^+) & [J(\rho)X(\bar{T},\rho)+B_d(\rho)U_d(\rho)]^{\T}\\
       \star & -X(\bar{T},\rho)
     \end{bmatrix}\preceq0
    \end{equation}
    hold for all $(\tau,\rho,\rho^+,\mu)\in[0, \bar{T}]\times\mathcal{P}\times\mathcal{P}\times \mathcal{D}^v$. Then, the LPV system \eqref{eq:mainsyst2}-\eqref{eq:K_MDT} with parameter trajectories in $\mathscr{P}_{\hspace{-1mm}{\scriptscriptstyle\geqslant \bar{T}}}$ is uniformly exponentially stable with the controller gains $K_c(\tau,\rho)=U_c(\tau,\rho)X(\tau,\rho)^{-1}$ and $K_d(\rho)=U_d(\rho)X(\bar T,\rho)^{-1}$.\hfill\mendth
\end{theorem}
\begin{proof}
The state matrix of the continuous-time part of the closed-loop system \eqref{eq:mainsyst2}-\eqref{eq:K_MDT} is given by
  \begin{equation}\label{eq:clglab}
A_{cl}(\tau,\rho)=\left\{\begin{array}{ll}
    A(\rho)+B_c(\rho)K_c(\tau,\rho)&, \tau\in[0,\bar T]\\
    A(\rho)+B_c(\rho)K_c(\bar T,\rho)&, \tau\in(\bar T,T_k],
\end{array}\right.
\end{equation}
while the state-matrix of the discrete-time part is given by $J_{cl}(\rho):=J(\rho)+B_d(\rho)K_d(\rho)$. \bll{Note also that the matrix $P(\tau,\rho)$ is invertible since it is positive definite as the function $x^{\T}P(\tau,\rho)x$ is positive definite; see e.g. the proof of Theorem \ref{th:cstDT}.}  Then, substituting this matrix in the conditions of Theorem \ref{th:minDT} and performing a congruence transformation with respect to the matrix $X(\tau,\rho)=P(\tau,\rho)^{-1}$ yields the conditions \eqref{eq:stabzCT1}-\eqref{eq:stabzCT2} where we have used the change of variables  $U_c(\tau,\rho)=K_c(\tau,\rho)X(\tau,\rho)$ and the facts that $-\partial_\tau X(\tau,\rho)=X(\tau,\rho)\partial_\tau P(\tau,\rho)X(\tau,\rho)$ and $-\partial_{\rho_i} X(\tau,\rho)=X(\tau,\rho)\partial_{\rho_i} P(\tau,\rho)X(\tau,\rho)$. Finally, the condition \eqref{eq:stabzCT3} is obtained from \eqref{eq:minDT:3} using successively a Schur complement, a congruence transformation with respect to $\diag(X(0,\rho),I)$ and the change of variables $U_d(\rho)=K_d(\rho)X(\bar{T},\rho)$. This proves the desired result.
\end{proof}

As for the previously obtained results, the conditions in Theorem \ref{th:minDT_CT} can be checked using convex SOS programming since the conditions are convex in the decision variables $X$, $U_c$ and $U_d$. We can also see that the chosen structure for the continuous-time controller matches perfectly the stability conditions and this matching allows us to use elementary and well-known linearizing change of variables. Should we have chosen a different structure, finding a linearizing change of variables would have been way more challenging. A difficulty that may arise when implementing timer-dependent controllers is that the initial timer value may be unknown and there my be a mismatch between the actual timer value and the implemented one. However, this will only result in an unstable initial transient phase of finite duration but which will not affect the long-term stability properties of the system as the system is linear. Another difficulty with the use of timer-dependent controllers is that they cannot be easily discretized, especially when the sensitivity of the gain with respect to the timer is high. As the timer has unit-rate, we can solve this problem by choosing a sufficiently small sampling period and/or a small minimum dwell-time value for the design. Other solutions, addressed in the next section and in Section \ref{sec:stabz2}, consist of considering timer-independent and sampled-data controllers, respectively.

\subsection{Stabilization by timer-independent controllers}

\bl{Due to the possible difficulty of implementing timer-dependent controllers, we suggest to design timer-independent ones taking the form
\begin{equation}\label{eq:K_MDT2}
\begin{array}{rcl}
  u_c(t)&=&K_c(\rho(t))x(t),\ t\ne t_k,k\ge0\\
  u_d(k)&=&K_d(\rho(t_k))x(t_k),\ k\ge1
\end{array}
\end{equation}
where $K_c(\cdot)\in\mathbb{R}^{m_c\times n}$ and $K_d(\cdot)\in\mathbb{R}^{m_d\times n}$ are the gains of the controllers. Even though the structure of this controller is simpler than that of the timer-dependent controller, its design is more cumbersome as its structure is not adapted to that of the clock-dependent stability conditions. Indeed, the difficulty in the design of such a controller lies in the fact that, in Theorem \ref{th:minDT_CT}, the gain $K_c$ of the controller is obtained using the expression $K_c(\tau,\rho)=U_c(\tau,\rho)X(\tau,\rho)^{-1}$. It is immediate to see that a timer-independent controller can be obtained by setting both $U_c(\tau,\rho)$ and $X(\tau,\rho)$ to be timer-independent. However, making this simplification would not result in a stability condition under minimum dwell-time but in a stability result under arbitrary dwell-time, which is unfeasible in most scenarios. In this regard, it is not reasonable to set the Lyapunov function to be timer-independent. Adding the constraint that $\partial\tau K_c(\tau,\rho)=0$ for all $\rho\in\mathcal{P}$ is definitely another interesting option but this constraint is clearly nonconvex and, therefore, difficult to consider. To the best of the author's knowledge, this has never been considered. A rather well-known solution to this problem is the transfer of the timer-independent constraint to another decision variable than the Lyapunov matrix. This can be possibly achieved through the use of so-called slack-variables \cite{Geromel:98,Oliveira:99,Tuan:03,Wu:10,Briat:book1,Ebihara:15}, which leads to the following result:
\begin{theorem}[Minimum dwell-time - Timer-independent conditions]\label{th:minDT_CTrelax}
\bll{Assume that the matrices of the system \eqref{eq:mainsyst2} satisfy Assumption \ref{hyp:AB}, and let $\bar T\in\mathbb{R}_{>0}$ be given}. Assume further that there exist a continuously differentiable matrix-valued function $Q:[0, \bar{T}]\times\mathcal{P}\mapsto \mathbb{S}^n$, $Q(\bar T,\rho)\in\mathbb{S}^n_{\succ0}$, $\rho\in\mathcal{P}$,  a matrix-valued function $U_c:\mathcal{P}\mapsto \mathbb{R}^{m_c\times n}$,  a matrix-valued function $U_d:\mathcal{P}\mapsto \mathbb{R}^{m_d\times n}$, and a scalar $\eps>0$ such that the conditions
  \begin{equation}\label{eq:minDT:2relax:stabz}
       \begin{bmatrix}
         -(Y+Y^{\T}) & A(\tau,\rho)Y+B_c(\tau,\rho)U_c(\rho)-Y^{\T}+Q(\tau,\rho)\\
         \star & \partial_\tau X(\tau,\rho)+\sum_{i=1}^N\partial_{\rho_i} X(\tau,\rho)\mu_i+\He[A(\tau,\rho)Y+B_c(\tau,\rho)U_c(\rho)+\eps I]
       \end{bmatrix}\preceq0
    \end{equation}
     \begin{equation}\label{eq:minDT:1relax:stabz}
     \begin{bmatrix}
         -(Y+Y^{\T}) & A(\bar{T},\rho)Y+B_c(\bar{T},\rho)U_c(\rho)-Y^{\T}+Q(\tau,\rho)\\
         \star & \sum_{i=1}^N\partial_{\rho_i}X(\bar{T},\rho)\mu_i+\He[A(\bar{T},\rho)Y+B_c(\bar{T},\rho)K(\rho)]+\eps I
       \end{bmatrix}\preceq0
    \end{equation}
    and
     \begin{equation}\label{eq:minDT:3relax:stabz}
     \begin{bmatrix}
       Q(0,\rho^+)-(Y+Y^{\T}) & J(\rho)Y+B_d(\rho)U_d(\rho)-Y^{\T}\\
        \star & -Q(\bar{T},\rho)+\He[J(\rho)Y+B_d(\rho)U_d(\rho)]
     \end{bmatrix}\preceq0
       \end{equation}
    hold for all $(\tau,\rho,\rho^+,\mu)\in[0, \bar{T}]\times\mathcal{P}\times\mathcal{P}\times \mathcal{D}^v$.  Then, the LPV system \eqref{eq:mainsyst2}-\eqref{eq:K_MDT2} with parameter trajectories in $\mathscr{P}_{\hspace{-1mm}{\scriptscriptstyle\geqslant \bar{T}}}$ is uniformly exponentially stable with the controller gains $K_c(\rho)=U_c(\rho)Y^{-1}$ and $K_d(\rho)=U_d(\rho)Y^{-1}$.
\hfill\mendth
\end{theorem}
\begin{proof}
  The proof is based on the substitution of the matrices of the closed-loop system in the conditions of Theorem \ref{th:minDTrelax} with the simplification $X_1=X_2=Z_1=Z_2=X$  for some constant matrix $X\in\mathbb{R}^{n\times n}$. Performing a congruence transformation on the conditions with respect to $\diag(Y,Y)$ where $Y=X^{-1}$ and letting $U_c(\rho)=K_c(\rho)Y$, $U_d(\rho)=K_d(\rho)Y$, and $Q(\cdot,\cdot)=Y^{\T}P(\cdot,\cdot)Y$  yield the conditions.
\end{proof}

It is worth mentioning that the choice that $X_1=X_2=Z_1=Z_2=X$ allows us to obtain convex design conditions at the expense of some conservatism. Other possibilities exist but would introduce extra scalar parameters to be tuned manually, which would have resulted in a more complicated design procedure. However, we will see in the example of the next section that the above theorem may still yield interesting results. It is also worth mentioning that other constraints can be enforced such as a block-structured controller can be obtained by enforcing the matrix $X$ to have the same block-structure (which should be stable by inversion). A simple example is the case of block diagonal controller where $U_c$ and $X$ are set to be block diagonal. Robust (i.e. parameter independent) controllers can also be designed using this method by restricting the matrices $U_c$ and $U_d$ to be constant matrices.}

\subsection{Example}

Let us consider back the example from \citep{Briat:15d}
\begin{equation}\label{eq:systex_stabz}
  \dot{x}=\begin{bmatrix}
    3-\rho & 1\\1-\rho &\ \ \ 2+\rho
  \end{bmatrix}x+\begin{bmatrix}
    1\\
    1+\rho
  \end{bmatrix}u_c,\ J=I_n
\end{equation}
where $\mathcal{P}=[0,\rho_{\textnormal{max}}]$, $\rho_{\textnormal{max}}=1$, and $\mathcal{D}=[-\nu,\nu]$. It was proved in \citep{Briat:15d} that this system cannot be stabilized quadratically. This latter property makes it a perfect example to illustrate the proposed approach since neither quadratic nor robust stabilization results can be used here. Applying then Theorem \ref{th:minDT_CT} with $\bar T=0.05$, we find that the conditions are feasible for $\nu\in\{0,0.1,0.3\}$ for $d=2$ and $\nu\in\{0.5,0.7,0.9,1,2\}$ for $d=3$. When $d=2$ the number of primal/dual variables is 834/180 whereas, when $d=3$, this number is 2414/315. Finally, when $d=2$, it takes roughly 2.62sec to solve the problem whereas, in the case $d=3$, it takes around 6.31sec. 
For simulation purposes, we consider the  parameter trajectory
\begin{equation}\label{eq:param_phase}
  \rho(t_k+\tau)=\dfrac{\rho_{\textnormal{max}}}{2}\left[1+\sin\left(\dfrac{2\nu (t_k+\tau)}{\rho_{\textnormal{max}}}+\varphi_k\right)\right],\ k\in\mathbb{Z}_{\ge0}
\end{equation}
where $\varphi_k$ is a uniform random variable taking values in $[0,2\pi]$ and we generate a random sequence of instants satisfying the minimum dwell-time condition. At each of these time instants, we draw a new value for the random variable  $\varphi_k$, which introduces a discontinuity in the trajectory. Note, however, that between discontinuities we have that $|\dot{\rho}(t_k+\tau)|\le\nu$ for all $\tau\in(0,T_k]$, $k\in\mathbb{Z}_{\ge0}$. We then obtain the results depicted in Fig.  \ref{fig:mySystemCT} for the case $d=3$, $\rmax=1$, and $\nu=1$ where we can see that stabilization is indeed achieved for this system.

\bl{We now use Theorem \ref{th:minDT_CTrelax} for the design of a timer-independent controller. Due to the conservatism of the relaxed conditions, it is not possible to stabilize the system for such a broad range of constraints as in the timer-dependent case but it is possible to find a stabilizing controller in the case where $\bar T=0.05$, $d=2$, $\nu=1$, and $\rho_{\textnormal{max}}=0.6$. Note that the LMI constraint \eqref{eq:minDT:3relax} is changed here to $\tilde{S}(0,\rho^+)-\tilde{S}(\bar{T},\rho)\preceq0$ here since the state does not jump. The obtained matrices are given
\begin{equation}\label{eq:controller_CT_Relax1}
  Y=\begin{bmatrix}
   10.028 & 30.292\\
   15.114 & 52.193
  \end{bmatrix}, K(\rho)=\begin{bmatrix}
    -19.455-8.4526\rho-19.131\rho^2\\
      7.7740+3.9694\rho+9.1796\rho^2
  \end{bmatrix}^{\T}, X(\tau,\rho)=[X_{ij}(\tau,\rho)]_{i,j=1,2}
\end{equation}
where
\begin{equation}
  \begin{array}{rcl}
    X_{11}(\tau,\rho)&=&  29.962+13.055\tau+0.43986\rho+0.34533\tau^2-19.903\tau\rho+0.044874\rho^2\\
    X_{12}(\tau,\rho)&=& 44.005+15.763\tau+0.62755\rho-0.66049\tau^2-24.719\tau\rho+0.10091\rho^2\\
    X_{22}(\tau,\rho)&=& 81.148+19.004\tau+0.59441\rho-0.072018\tau^2-22.198\tau\rho+0.31172\rho^2.
  \end{array}
\end{equation}
The simulation results are depicted in Figure \ref{fig:mySystemCT_Relax1} where we can observe that the controller stabilizes the system but that it takes longer to reach equilibrium. Note also that the fact that we can find a timer-independent controller may be due to the fact that the minimum dwell-time is small and, hence, the range of values for $\tau$ is small and, as a result, we can approximate fairly well a timer-dependent controller by a timer-independent one. The following case shows that the problem remains feasible for larger dwell-times. Indeed, the problem is still solvable with $\bar{T}=1$, $d=2$, $\nu=1$, $\rho_{\textnormal{max}}=0.6$, and the matrices

\begin{equation}\label{eq:controller_CT_Relax2}
  Y=\begin{bmatrix}
   7.6849& 23.045\\
   11.582 &39.644
  \end{bmatrix}, K(\rho)=\begin{bmatrix}
 -19.233-10.112\rho-15.726\rho^2\\
  7.6689+4.8734\rho+7.3227\rho^2
  \end{bmatrix}^{\T}, X(\tau,\rho)=[X_{ij}(\tau,\rho)]_{ij}
\end{equation}
where
\begin{equation}
  \begin{array}{rcl}
    X_{11}(\tau,\rho)&=&  22.041+1.5769\tau+0.70604\rho+0.20096\tau^2-0.97129\tau\rho-0.30546\rho^2\\
    X_{12}(\tau,\rho)&=& 32.532+2.0088\tau+1.1276\rho+0.2084\tau^2-1.1802\tau\rho+0.57069\rho^2\\
    X_{22}(\tau,\rho)&=& 60.359+2.8099\tau+1.6351\rho+0.1849\tau^2-1.3145\tau\rho+1.9863\rho^2.
  \end{array}
\end{equation}
The simulation results are depicted in Figure \ref{fig:mySystemCT_Relax2} where we can observe that the controller indeed stabilizes the system.}

 \begin{figure}
    \centering
    \includegraphics[width=0.8\textwidth]{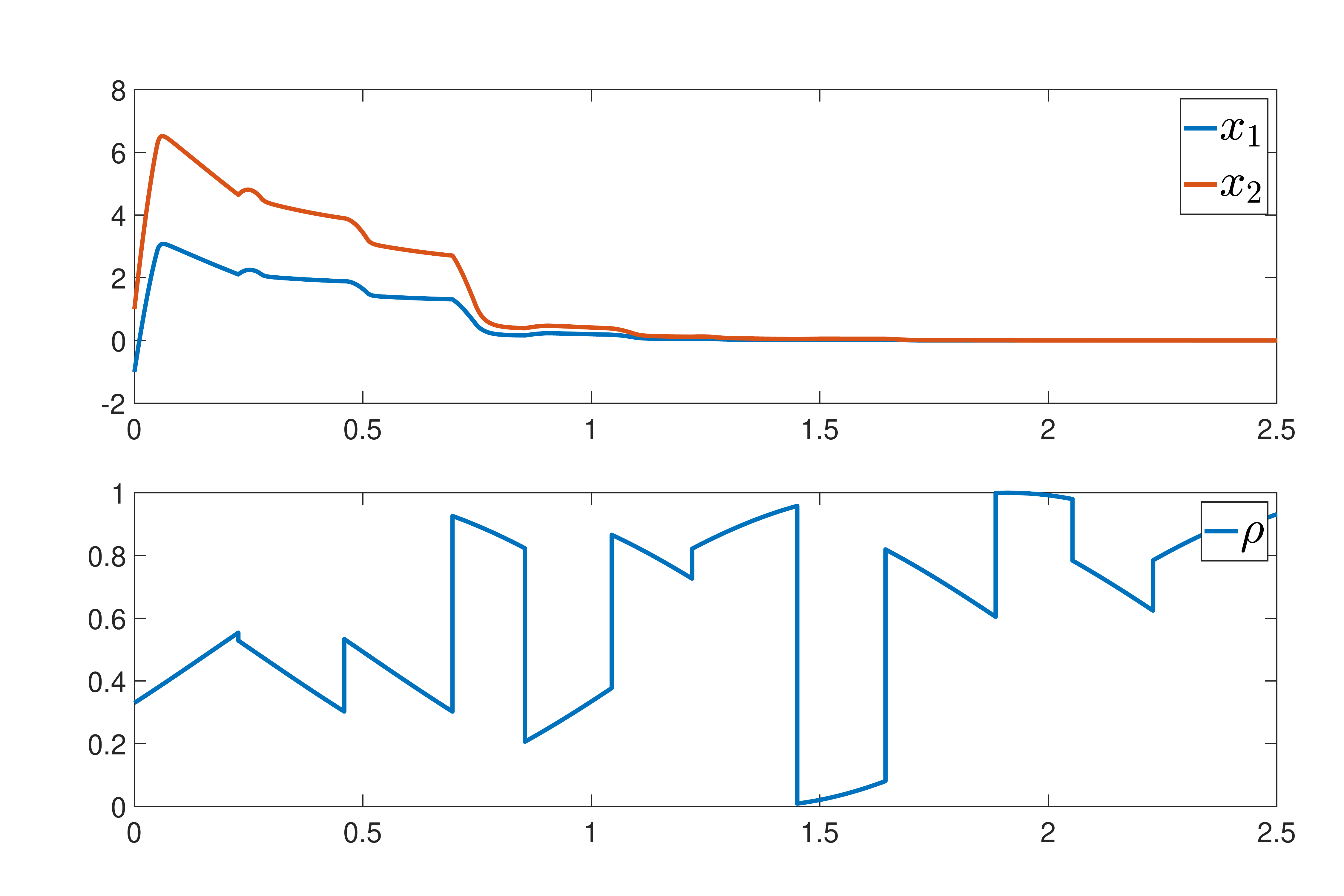}
    \caption{Trajectories of the closed-loop system \eqref{eq:systex_stabz}-\eqref{eq:K_MDT}-\eqref{eq:param_phase} with $\nu=1$, $\rmax=1$ and $\bar{T}=0.05$.}\label{fig:mySystemCT}
  \end{figure}

 \begin{figure}
    \centering
    \includegraphics[width=0.8\textwidth]{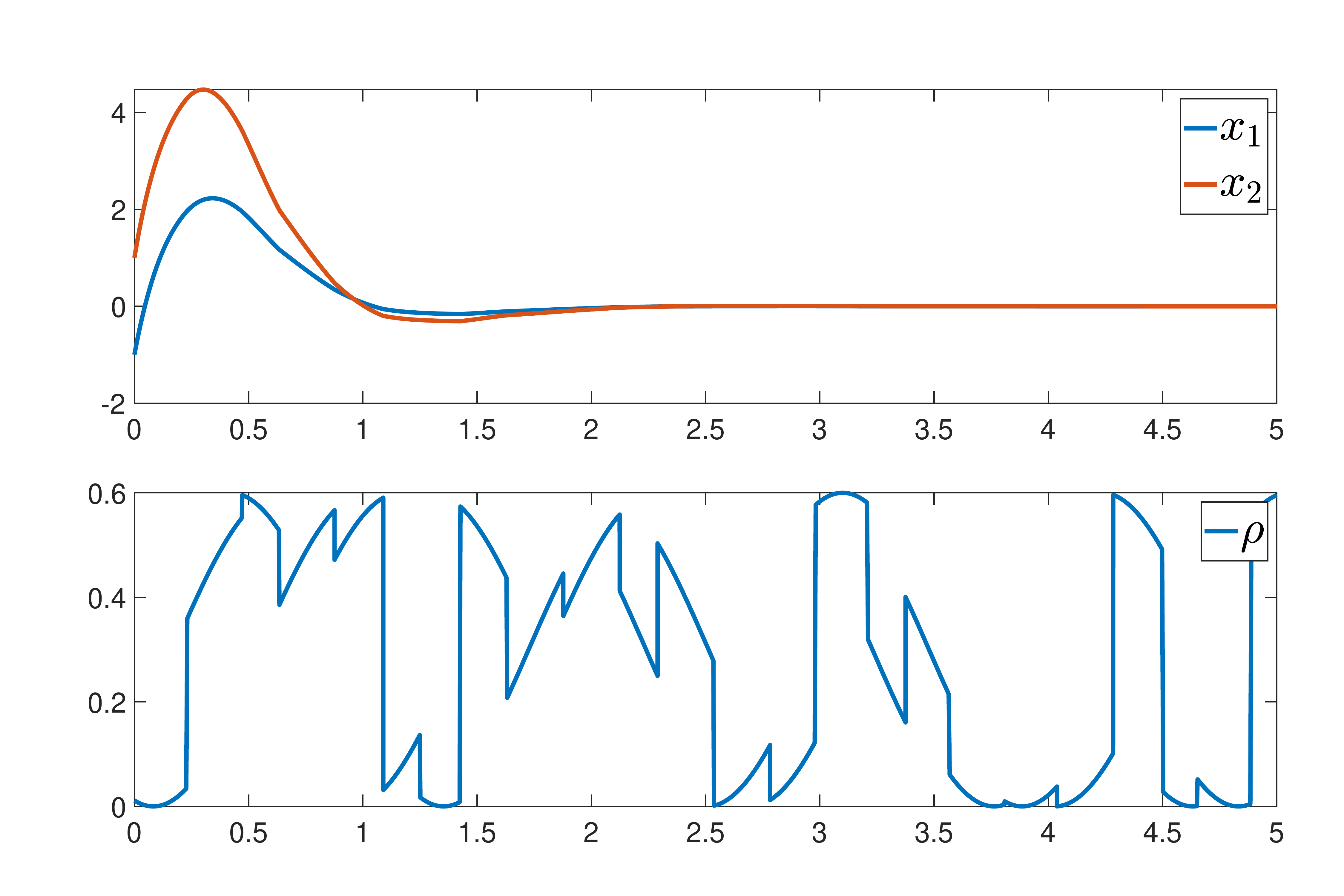}
    \caption{Trajectories of the closed-loop system \eqref{eq:systex_stabz}-\eqref{eq:K_MDT}-\eqref{eq:param_phase}-\eqref{eq:controller_CT_Relax1} with $\nu=5/3$ and $\rmax=3/5$ and  $\bar{T}=0.05$.}\label{fig:mySystemCT_Relax1}
  \end{figure}

 \begin{figure}
    \centering
    \includegraphics[width=0.8\textwidth]{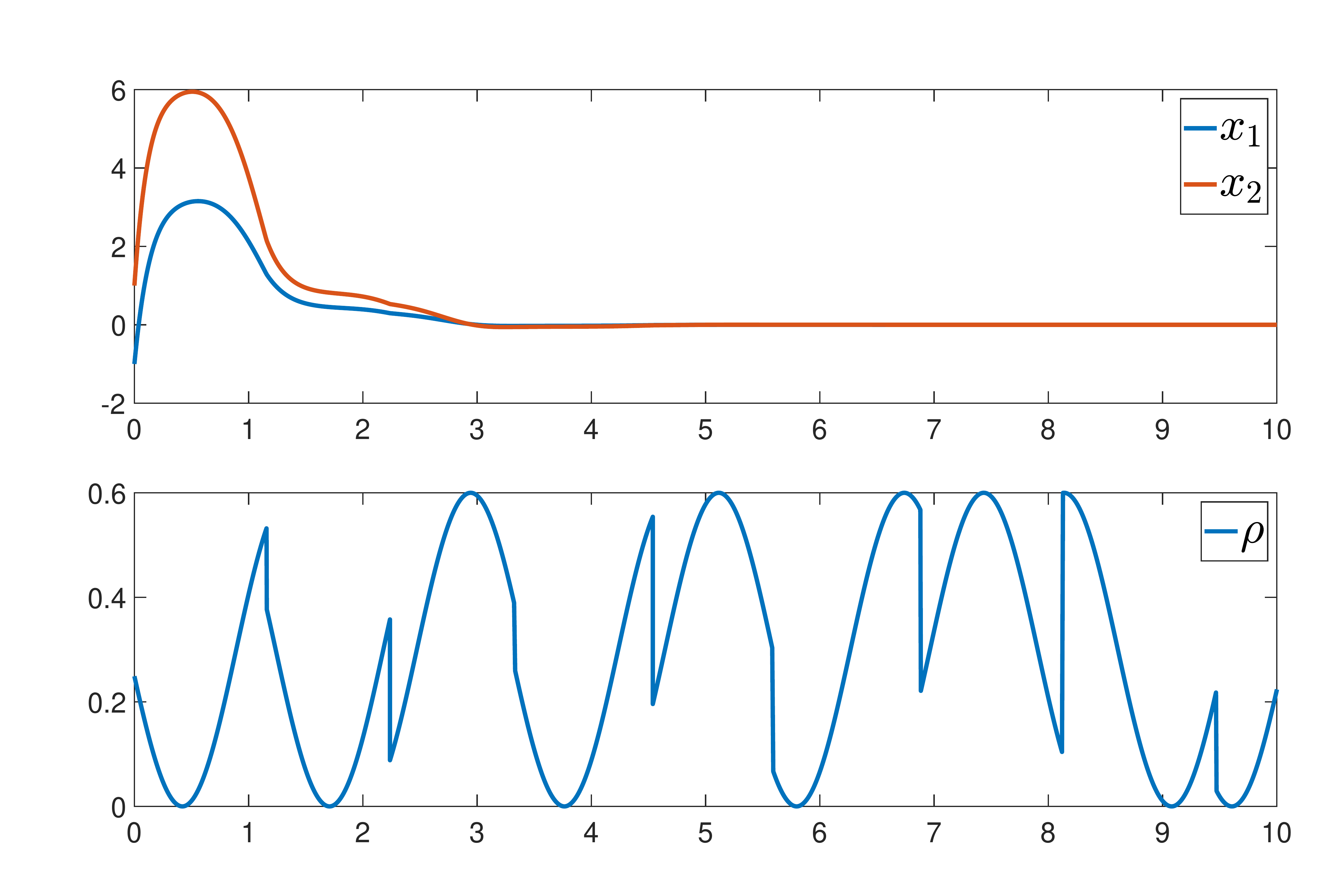}
     \caption{Trajectories of the closed-loop system \eqref{eq:systex_stabz}-\eqref{eq:K_MDT}-\eqref{eq:param_phase}-\eqref{eq:controller_CT_Relax2} with $\nu=5/3$ and $\rmax=3/5$ and  $\bar{T}=1$.}\label{fig:mySystemCT_Relax2}
  \end{figure}

%

\section{Range dwell-time stability analysis and stabilization using sampled-data LPV controllers}\label{sec:stabz2}

\bll{The objective of this section is the development of a stability condition under a range dwell-time constraint and its use in the design of LPV sampled-data controllers. Examples are also given for illustration}

\subsection{Range dwell-time stability analysis}
We consider in this section the following class of parameter trajectories:
\begin{equation}\label{eq:rangeDTparam}
 \mathscr{P}_{\hspace{-1mm}{\scriptscriptstyle[\Tmin,\Tmax]}}:=\left\{\begin{array}{c}
    \rho:\mathbb{R}_{\ge0}\mapsto \mathcal{P}\left|\begin{array}{c}
     \dot{\rho}(t)\in\mathcal{Q}(\rho(t)),t\in(t_k,t_{k+1}],T_k\in[\Tmin,\Tmax]\\
    t_0=0,\rho(t_0)=\rho(t_0^+),k\in\mathbb{Z}_{\ge0}
    \end{array}\right.\end{array}\right\}
\end{equation}
where $T_k:=t_{k+1}-t_k$ and $0<\Tmin\le\Tmax$. With this definition in mind, we can now state the stability result under range dwell-time:
\begin{theorem}[Range dwell-time]\label{th:rangeDT_DT}
\bll{Assume that the parameter trajectories satisfy Assumption \ref{hyp:param} and let $0<\Tmin \le \Tmax <\infty$ be given.} Assume further that there exist a continuously differentiable matrix-valued function $P:[0, \Tmax ]\times\mathcal{P}\mapsto \mathbb{S}^n$, $P(0,\rho)\succ0$, $\rho\in\mathcal{P}$, and a scalar $\eps>0$ such that the conditions 
     \begin{equation}\label{eq:rdt:1}
      -\partial_{\varsigma} P(\varsigma,\rho)+\sum_{i=1}^N\partial_{\rho_i} P(\varsigma,\rho)\mu_i+\He[P(\varsigma,\rho)A(\rho)]\preceq0
    \end{equation}
    and
     \begin{equation}\label{eq:rdt:2}
      J(\rho)P(\xi,\rho^+)J(\rho)-P(0,\rho)+\eps I_n\preceq0
    \end{equation}
    hold for all $(\rho,\rho^+)\in\mathcal{P}\times\mathcal{P}$, all $\mu\in\mathcal{D}^v$, all $\varsigma\in[0, \Tmax ]$ and all $\xi\in[\Tmin ,\Tmax ]$. Then, the impulsive LPV system \eqref{eq:mainsyst} with parameter trajectories in $\mathscr{P}_{\hspace{-1mm}{\scriptscriptstyle[\Tmin,\Tmax]}}$ is uniformly exponentially stable under the range dwell-time constraint $[\Tmin ,\Tmax ]$.

      \bll{Moreover, when the above conditions hold, then there exists a matrix-valued function $Q:\mathcal{P}\mapsto\mathbb{S}^n_{\succ0}$ such that
    \begin{equation}
            M^{\T}Q(\rho^+)M-Q(\rho)\prec0
    \end{equation}
holds for all $\textstyle M\in\bigcup_{T\in[\Tmin,\Tmax]}\mathcal{M}(T,\rho)$ and all $(\rho,\rho^+)\in\mathcal{P}\times\mathcal{P}$. and the pre-jump embedded discrete-time system \eqref{eq:mainsystDT1} is uniformly exponentially stable.}\hfill\mendth
\end{theorem}
\begin{proof}
The proof of this result follows the same lines as the proof of Theorem \ref{th:minDT} with the difference that we consider here the Lyapunov function
\begin{equation}
  V(x,\tau,\rho,T)=x^{\T}P(T-\tau,\rho)x
\end{equation}
where $P(\cdot,\cdot)$ satisfies the conditions of Theorem \ref{th:rangeDT_DT}. Pre- and post-multiplying \eqref{eq:rdt:1} by $x(t_k+\tau)^{\T}$ and $x(t_k+\tau)^{\T}$ and letting $\varsigma=T_k-\tau$, $\rho=\rho(t_k+\tau)$,  $\mu=\dot{\rho}(t_k+\tau)$,  and integrating from 0 to $T_k$ yields
\begin{equation}
 V(x(t_{k+1}),T_k,\rho(t_{k+1}),T_k)-V(x(t_{k}^+),0,\rho(t_{k}^+),T_k)\le0
\end{equation}
or, equivalently,
\begin{equation}
 x(t_{k})^{\T}J(\rho(t_k))^{\T}\left[\Phi_\rho(t_{k+1},t_k)^{\T}P(0,\rho(t_{k+1}))\Phi_\rho(t_{k+1},t_k)-P(T_k,\rho(t_{k}^+))\right]J(\rho(t_k))x(t_{k})\le0
\end{equation}
which together with \eqref{eq:rdt:2} implies that with
\bll{\begin{equation}
 x(t_{k})^{\T}\left[M_k^{\T}P(0,\rho(t_{k+1}))M_k-P(0,\rho(t_k))\right]x(t_{k})\le-\eps||x(t_k)||_2^2
\end{equation}
holds for all $x(t_{k})\in\mathbb{R}^n$, all $\textstyle M_k\in\bigcup_{T\in[\Tmin,\Tmax]}\mathcal{M}(T,\rho(t_k))$, $(\rho(t_k),\rho(t_{k+1}))\in\mathcal{P}\times\mathcal{P}$. This proves the existence of a matrix-valued function $Q:\mathcal{P}\mapsto\mathbb{S}^n_{\succ0}$ that verifies the concluding statement of the theorem and that the embedded discrete-time system \eqref{eq:mainsystDT1} is uniformly exponentially stable under range dwell-time $[\Tmin ,\Tmax ]$. Using the exact same arguments as in the proof of the previous results, it can be shown that this also implies that the LPV system \eqref{eq:mainsyst} with parameter trajectories in $\mathscr{P}_{\hspace{-1mm}{\scriptscriptstyle[\Tmin,\Tmax]}}$ following the same arguments as in the previous results. The proof is completed.}
\end{proof}

It seems important to mention the presence of the negative term in front of the timer-derivative. This comes from the fact that we consider here the Lyapunov function  $V(x,\tau,\rho,T)=x^{\T}P(T-\tau,\rho)x$, which introduces a negative sign when deriving it by $\tau$. It is interesting to note that the timer variable still measures the time elapsed since the last jump but that the Lyapunov matrix depends on the remaining time $T_k-\tau$ until the next jump. This is different from the constant and minimum dwell-time stability and stabilization results where the Lyapunov matrices directly depended on the time elapsed since the last jump. Using such a "reverse timer condition" will allow us to obtain stabilization conditions using a dwell-time-independent sampled-data controller without the need for using relaxed conditions and slack variables. \bll{This will be further explained in the next section.}

\subsection{Sampled-data stabilization of LPV systems}

\bll{We consider in this section the following class of sampled-data gain-scheduled state-feedback controllers
\begin{equation}\label{eq:K_MDT3}
u(t_k+\tau)=K_1(\rho(t_k))x(t_k)+K_2(\rho(t_k))u(t_k)
\end{equation}
where $\tau\in(0,T_k]$, $T_k\in[\Tmin ,\Tmax ]$ and where $K_1(\cdot)\in\mathbb{R}^{m\times n}$ and $K_2(\cdot)\in\mathbb{R}^{m\times m}$ are the gains of the controller to be determined. This controller is a bit more general than those usually considered in the literature \cite{Ramezanifar:12,GomesdaSilva:15,Gomes:18} due to the presence of the extra term $K_2$, which makes the control law a filtered one and may lead to less conservative results than a non-filtered controller having $K_2\equiv0$, at the price of a slight increase of the computational and the implementation complexity. As the control input involves the parameter $\rho(t_k)$ over the interval $(t_k,t_{k+1}]$, this controller is irrelevant to consider in the case of discontinuous parameter trajectories since, in such a case, the system would evolve according to a parameter trajectory starting at $\rho(t_k^+)$ , which is independent of $\rho(t_k)$, over the same interval. Note, however, that if the controller matrix $K$ would depend on both $\rho$ and $\rho^+$, the consideration of discontinuous parameter trajectories would make perfect sense. Unfortunately, the current setting is that parameter values are only measured at the time instants $t_k$ and the value $\rho(t_k^+)$ is inaccessible to the controller, so this option is ruled out. This mismatch leads us to consider the family of parameter trajectories
\begin{equation}
  \mathscr{P}_{\hspace{-2pt}\infty}:=\left\{
    \rho:\mathbb{R}_{\ge0}\mapsto \mathcal{P}\ \left|\
      \dot{\rho}(t)\in\mathcal{Q}(\rho(t)),\ t\ge0
    \right.\right\}.
\end{equation}
and the family of sampling instants
\begin{equation}
 \mathscr{T}:=\left\{\{t_k\}_{k\in\mathbb{Z}_{\ge0}}\left|\begin{array}{c}
      T_k:=t_{k+1}-t_k\in[\Tmin ,\Tmax ],t_0=0,\ k\in\mathbb{Z}_{\ge0}
    \end{array}\right.\right\}
\end{equation}
where $0\le \Tmin \le \Tmax <\infty$. When $\Tmin=\Tmax=\bar T$, we recover the well-known periodic sampling case.}

One of the advantages of the impulsive framework is that it can exactly represent sampled-data systems by augmenting the state of the system with the control input as follows. On the strength of this fact, the closed-loop system \eqref{eq:mainsyst2}-\eqref{eq:K_MDT3} can be written as
\bll{\begin{equation}\label{eq:mainsystH_SD}
\begin{array}{rcl}
\begin{bmatrix}
  \dot{x}(t)\\
  \dot{u}(t)
\end{bmatrix}&=& A_a(\rho(t))\begin{bmatrix}
  x(t)\\
  u(t)
\end{bmatrix},\ t\in(t_k,t_{k+1}],k\in\mathbb{Z}_{\ge0}\\
\begin{bmatrix}
  x(t_k^+)\\
  u(t_k^+)
\end{bmatrix}&=&J_a(\rho(t_k))\begin{bmatrix}
  x(t_k)\\
  u(t_k)
\end{bmatrix},\ k\in\mathbb{Z}_{>0},
  \end{array}
\end{equation}
where $J_a(\rho)=J_0(\rho)+B_0K(\rho)$ together with
\begin{equation}
  A_a(\rho):=\begin{bmatrix}
    A(\rho) & B(\rho)\\
    0 & 0
  \end{bmatrix},J_0(\rho):=\begin{bmatrix}
    J(\rho) & 0\\
    0 & 0
  \end{bmatrix},B_0:=\begin{bmatrix}
    0\\
    I_m
  \end{bmatrix},\ \textnormal{and }K(\rho):=\begin{bmatrix}
    K_1(\rho) & K_2(\rho)
  \end{bmatrix}.
\end{equation}}

We can now state the stabilization result of the section:
\begin{theorem}\label{th:rangeDT_DTstabz}
\bll{Assume that the parameter trajectories satisfy Assumption \ref{hyp:param} and let $0<\Tmin \le \Tmax <\infty$ be given. Assume further that there exist a continuously differentiable matrix-valued function $X:[0, \Tmax ]\times\mathcal{P}\mapsto \mathbb{S}^{n+m}$, $X(0,\rho)\succ0$, $\rho\in\mathcal{P}$, a matrix-valued function $U:\mathcal{P}\mapsto \mathbb{R}^{m\times (n+m)}$ and a scalar $\eps>0$ such that the conditions}
%
     \begin{equation}\label{eq:dskldksdlsdlkdlakdskl1}
      \partial_{\varsigma} X(\varsigma,\rho)-\sum_{i=1}^N\partial_{\rho_i} X(\varsigma,\rho)\mu_i+\He[A_a(\rho)X(\varsigma,\rho)]\preceq0
    \end{equation}
    and
     \begin{equation}\label{eq:dskldksdlsdlkdlakdskl2}
      \begin{bmatrix}
        -X(0,\rho)+\eps I_n & [J_0(\rho)X(0,\rho)+B_0U(\rho)]^{\T}\\
        \star & -X(\xi,\rho)
      \end{bmatrix}\preceq0
    \end{equation}
    hold for all $\rho\in\mathcal{P}$, all $\mu\in\mathcal{D}^v$, all $\varsigma\in[0, \Tmax ]$ and all $\xi\in[\Tmin ,\Tmax ]$.

\blue{Then, the sampled-data LPV system \eqref{eq:mainsyst2}-\eqref{eq:K_MDT3}  with parameter trajectories in $\mathscr{P}_{\hspace{-3pt}\infty}$ is uniformly exponentially stable under the range dwell-time condition $[\Tmin ,\Tmax ]$ (i.e. for all sequences of jumping instants in $\mathscr{T}$)  with the controller gain  \bll{$K(\rho)=U(\rho)X(0,\rho)^{-1}$.} \hfill\mendth}
\end{theorem}
\begin{proof}
  \bll{As the proof follows from similar arguments as the proof of Theorem \ref{th:minDT_CT}, it is only sketched. First of all, note that $P(\tau,\rho)$ is invertible for all $(\tau,\rho)\in[0,\Tmax]\times \mathcal{P}$ since it is positive definite (see the other proofs where the positivity definiteness of the function $x^{\T}P(\tau,\rho)x$ is proven). Performing a congruence transformation with respect to $X(\varsigma,\rho):=P(\varsigma,\rho)^{-1}$ on \eqref{eq:rdt:1}, where we have substituted $A$ by $A_a$, yields \eqref{eq:dskldksdlsdlkdlakdskl1}.  The sign inversions in the derivative terms comes from the fact that for any invertible matrix $R(s)$, we have that $\textstyle\frac{d}{ds}R(s)^{-1}=-R(s)^{-1}\left(\frac{d}{ds}R(s)\right)R(s)^{-1}$. Performing now a congruence transformation with respect to $X(0,\rho)$ on \eqref{eq:rdt:2}, where we have substituted $J$ by $J_a$, followed by a Schur complement and the change of variables $U(\rho)=K(\rho)X(0,\rho)$  yields a condition that is equivalent to \eqref{eq:dskldksdlsdlkdlakdskl2}. This proves the result.}
\end{proof}

\bll{In the previous section, it was anticipated that the use of reverse timer conditions would allow us to obtain convex stabilization conditions using a dwell-time-independent sampled-data controller without using any relaxed conditions nor slack variables. While the above result obviously demonstrates the correctness of this claim, we now discuss what would have happened, had we considered, for instance, Theorem \ref{th:cstDT} generalized to the range dwell-time case, i.e. $\bar T\in[\Tmin,\Tmax]$. In such a case, performing the same manipulations as in the proof of Theorem \ref{th:rangeDT_DTstabz} would have yielded a term of the form $K(\rho)P(\xi,\rho)^{-1}$ where $\xi\in[\Tmin,\Tmax]$. In the time-domain, $\xi$ can be substituted by either $T_{k}$ or $T_{k-1}$, which means that to obtain convex conditions, it would be necessary that $K$ be also a function of $\xi$. While it could be argued that the dwell-time could, in fact, be also used in the control law, the goal of this section was to obtain a controller that only depends on the parameter vector, which is interestingly a more delicate problems in the light of this remark and the above result.}

\bll{It seems also interesting to explain why it is difficult here to design a non-filtered controller using Theorem \ref{th:rangeDT_DTstabz}; i.e. a controller for which $K_2\equiv0$.  It was shown in \cite{Briat:13d}, in the context of the sampled-data control of LTI systems, that the stabilization conditions could be adapted in order to design such a controller. This was achieved by exploiting the fact that considering a block-diagonal upper-triangular Lyapunov matrix is both necessary and sufficient for proving the stability of the embedded discrete-time system associated with the sampled-data system. More concretely, the  following linear constraints $[X(0)]_{21}=0$ and $U_2=0$ were added to the stabilization conditions and the resulting  result was successfully used to design a non-filtered sampled-data controller. Unfortunately, this approach cannot be readily adapted to the present case as it would require the $(2,1)$ block of $X(0,\rho)$ to be zero, and a necessary condition for this constraint to be satisfied is that the $(2,1)$ block of $X(\tau,\rho)$ be independent of $\rho$, a constraint which is likely to dramatically increase the conservatism of the approach, ultimately leading to infeasible problems. A way around this problem would be to consider dilated LMI conditions for range dwell-time stability analysis in the same spirit as those obtained in Theorem \ref{th:minDT_CTrelax} in the minimum dwell-time case. This alternative way is not addressed any further in this paper for the sake of conciseness.
}

\subsection{Examples}

To illustrate the interest of the approach, we consider here three examples from the literature.

\begin{example}
  Let us consider back the system \eqref{eq:systex_stabz} with the difference that we now aim at stabilizing it with a gain-scheduled sampled-data state-feedback controller of the form \eqref{eq:K_MDT}. Solving for the sum of squares conditions associated with the conditions stated in Theorem \ref{th:rangeDT_DTstabz} with $d=2$, $\nu=1$, $\Tmin=0.01$ and $\Tmax=0.1$ yields the controller gain
  \begin{equation}\label{eq:controllldlsdsdlsd}
\begin{array}{rcl}
  K(\rho)&=&\dfrac{1}{\textnormal{den}(\rho)}\begin{bmatrix}
     3.01-2.00\rho+5.52\rho^2-2.43\rho^3-0.59\rho^4+0.69\rho^5+0.04\rho^6\\
    -0.74-0.29\rho+0.77\rho^2-1.13\rho^3+0.10\rho^4+0.24\rho^5+0.08\rho^6\\
    \hline
 -0.002+0.014\rho+0.029\rho^2-0.46\rho^3+1.10\rho^4-0.95\rho^5+0.28\rho^6
  \end{bmatrix}^{\T},\\
  \textnormal{den}(\rho)&=&-0.32+0.56\rho-1.20\rho^2+0.45\rho^3+1.15\rho^4-1.18\rho^5+0.23\rho^6.
\end{array}
\end{equation}
  Computational-wise, the underlying semidefinite program has 3078/525 primal/dual variables and is solved in 7.88sec. The trajectory of the closed-loop system is depicted in the top panel of Figure \ref{fig:MySystemCD} for the parameter trajectory $\rho(t)=(1+\sin(2\nu t))/2$ and initial condition $x_0=(-1,1)$, $u_0=0$. The dwell-time values have been randomly selected in the interval $[0.01,,0.1]$.
  \begin{figure}
    \centering
    \includegraphics[width=0.8\textwidth]{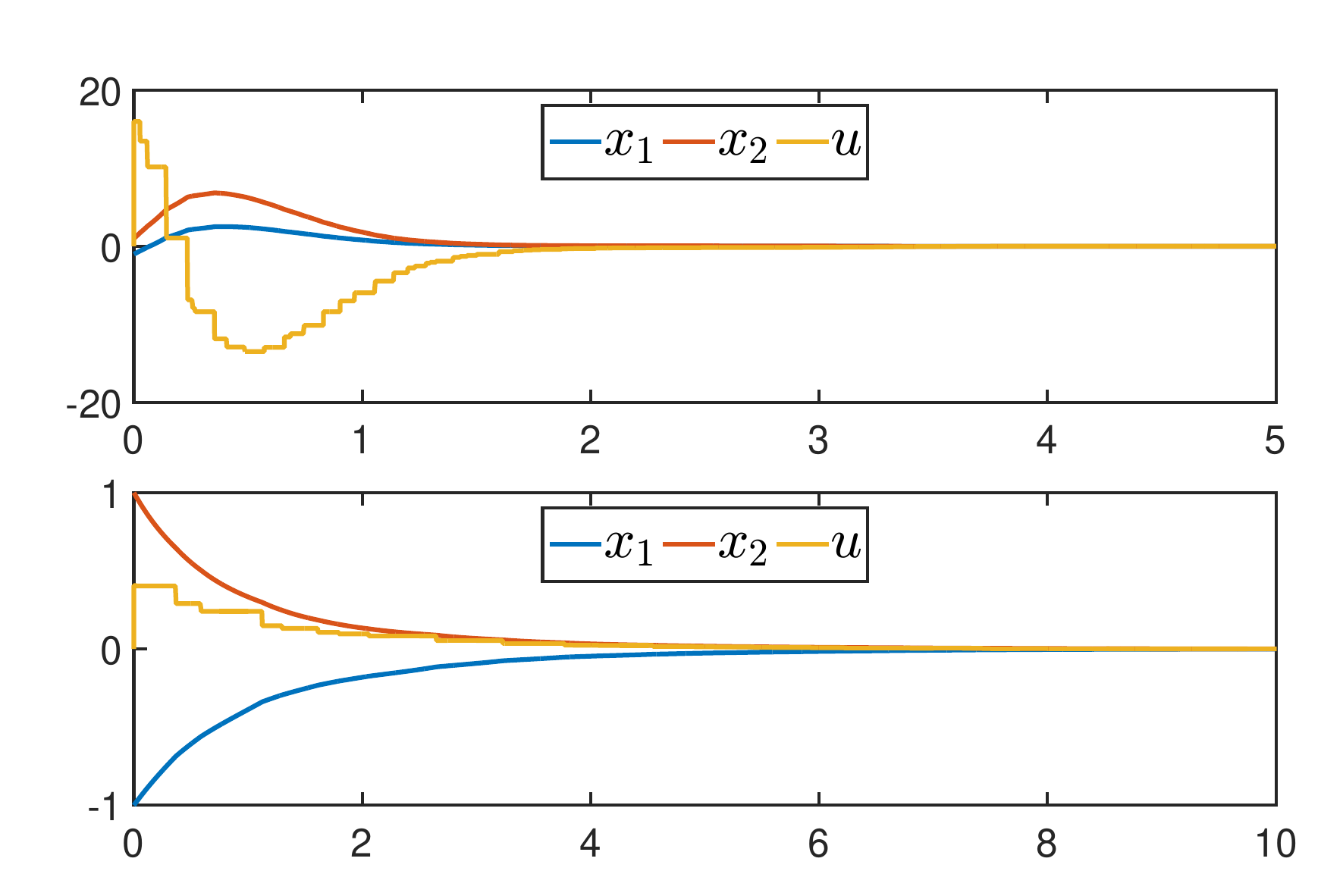}
    \caption{Trajectories of the closed-loop system \eqref{eq:systex_stabz}-\eqref{eq:K_MDT} (top) and \eqref{eq:Ramezanifar}-\eqref{eq:K_MDT} (bottom).}\label{fig:MySystemCD}
  \end{figure}
\end{example}

\begin{example}
  Let us consider the system \citep{Ramezanifar:12}
  \begin{equation}\label{eq:Ramezanifar}
    \dot{x}=\begin{bmatrix}
      2\rho & 1.1+\rho\\
      -2.2+\rho &-3.3+0.1\rho
    \end{bmatrix}x+\begin{bmatrix}
      2\rho\\
      0.1+\rho
    \end{bmatrix}u
  \end{equation}
  where $\rho(t)=\sin(0.2t)$, hence $\mathcal{P}=[-1,1]$ and $\mathcal{D}=[-0.2,0.2]$. It was shown in \citep{Ramezanifar:12} that this system could be stabilized at least up to $\Tmax =0.6$ using an input-delay model for the zero-order hold and a parameter-dependent Lyapunov-Krasovskii functional.
 Using polynomials of order 4 in the SOS conditions, we can solve the SOS program associated with Theorem \ref{th:rangeDT_DTstabz} and find a controller that makes the closed-loop system stable for all $T_k\in[0.001,0.6]$. The program has 9618/966 primal/dual variables and is solved in 36.26sec. The simulation results are depicted in Figure \ref{fig:MySystemCD} \bblue{where the dwell-time values were randomly selected in the interval $[0.001,0.6]$. The initial conditions are chosen to be $x_0=(-1,1)$, $u_0=0$.}
\end{example}

\begin{example}
  Let us consider now the system \citep{GomesdaSilva:15,Gomes:18}
  \begin{equation}\label{eq:Gomes}
    \dot{x}=\begin{bmatrix}
      0 & 1\\
      0.1 &0.4+0.6\rho
    \end{bmatrix}x+\begin{bmatrix}
      0\\
      1
    \end{bmatrix}u
  \end{equation}
   where $\rho(t)=\sin(\nu t)$. Hence, $\mathcal{P}=[-1,1]$ and $\mathcal{D}=[-\nu,\nu]$. Using a looped-functional approach, it was shown in \citep{GomesdaSilva:15} that, for $\Tmin =0.001$, this system could be stabilized up to $\Tmax =1.264$ when $\nu=0.2$ and up to $\Tmax =0.8$ when $\nu=1$. \bll{A refined approach from the same authors \cite{Gomes:18} yielded a value $\Tmax=1.349$.} Using Theorem \ref{th:rangeDT_DTstabz} with $d=4$, we can show that, for both $\nu=0.2$ and $\nu=1$, we can find a controller that stabilizes the system for all $T_k\in[0.001,1.3]$. \bl{In fact, the system remains stabilizable up to at least $\Tmax=2$ using the proposed approach.} The number of primal/dual variables is 9618/966 and the problem solves in approximately 25sec. For simulation purposes, we set $\Tmax =0.4$ for both $\nu=0.2$ and $\nu=1$, and we design controllers using Theorem  \ref{th:rangeDT_DTstabz} with $d=2$ (in this case, the number of primal/dual variables is given by 3078/525 and the problem is solved in 7sec). Using the initial conditions $x_0=(-1,1)$, $u_0=0$ \bblue{and random sequences of dwell-times in $[0,0.4]$}, we get the trajectories depicted in Figure \ref{fig:Gomes}. Note that, as in \citep{GomesdaSilva:15}, using a controller designed for $\Tmax =1.3$ would result in a very slow response for the closed-loop system which is not desirable.
      \begin{figure}
    \centering
    \includegraphics[width=0.8\textwidth]{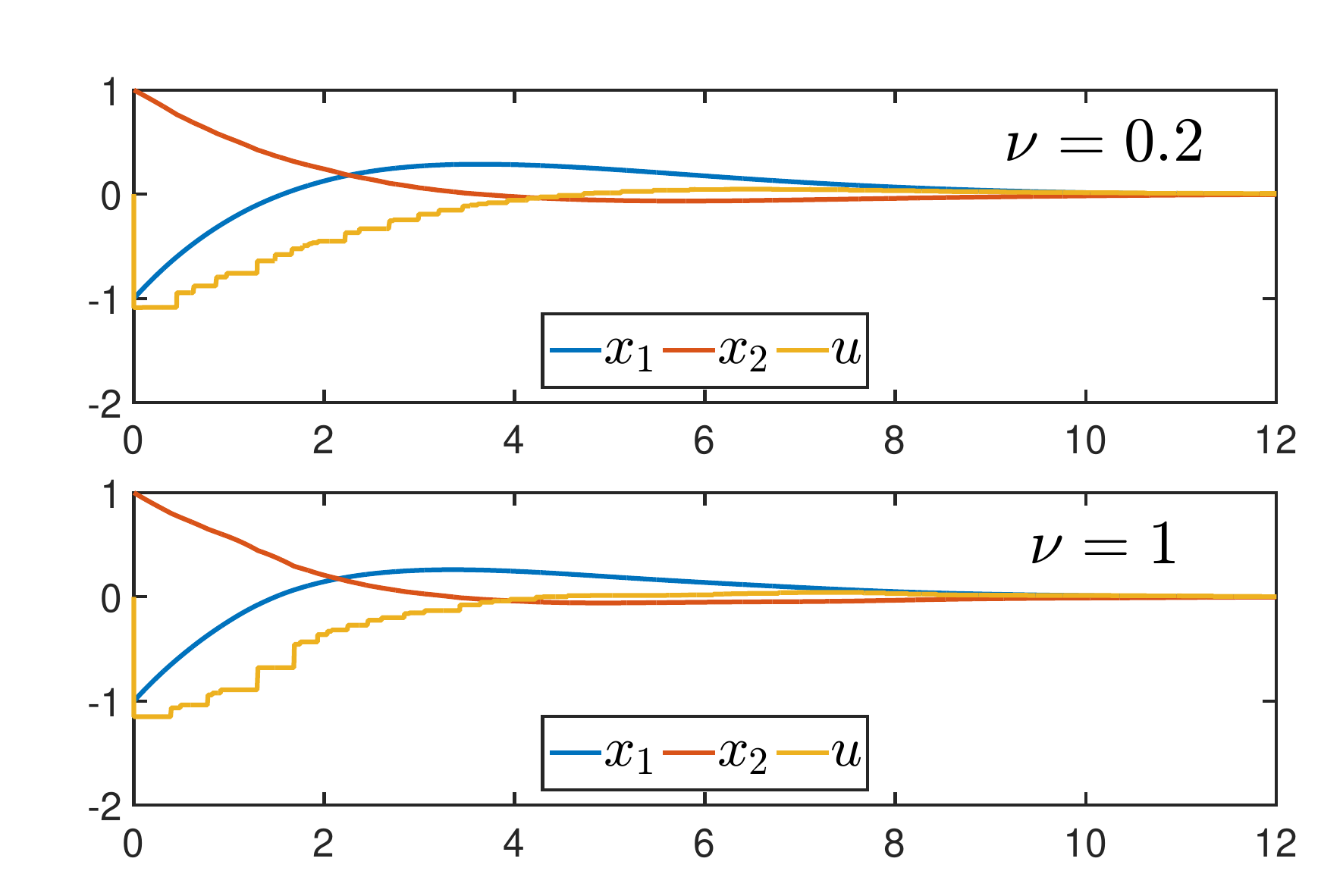}
    \caption{Trajectories of the closed-loop system \eqref{eq:Gomes}-\eqref{eq:K_MDT}.}\label{fig:Gomes}
  \end{figure}
\end{example}

\section{Discussion and Future Works}

\bll{The main drawback of the approach is certainly its complexity since the approach relies on infinite-dimensional semidefinite programs whose SDP approximations are typically large in size (i.e. in both the number of variables and the number constraints). This is the price to pay to obtain accurate stability and stabilization conditions. Indeed, computationally unfair comparative examples tend to suggest that the proposed approach leads to better results than previously obtained in the literature. However, it is expected in a near future to have improvements at the solver-level that will make this kind of approaches applicable to larger systems. Note, finally, that LMI methods are only restricted to small to medium size problems and they are in general not applicable to large systems unless the resulting semidefinite programs satisfy certain convenient structural properties, such as chordal sparsity \cite{Fujisawa:97,Zheng:19}, that can be exploited by the solver to reduce the complexity and speed-up the solving time. Potential extensions of the obtained results include the consideration of different types of Lyapunov functions such as polyhedral or homogeneous ones, and the consideration of additional clocks in order to consider multiple types of discrete events (such as control update and parameter discontinuities events). Converse results along the lines of \citep{Wirth:05,Goebel:12} for this class of systems could also be very interesting to obtain. Finally, the derivation of convex stabilization conditions for the design of dynamic output-feedback controller is a problem which is also worth investigating.}

{\small 

}


\end{document}